\documentclass[a4paper,12pt,pdf]{amsart}
\usepackage{enumerate, amsmath, amsfonts, amssymb, amsthm, thmtools, wasysym, graphics, graphicx, xcolor, frcursive,comment,bbm}

\usepackage{etex}

\usepackage[a4paper, margin=2.5cm]{geometry}

\usepackage{microtype}

\usepackage[hidelinks]{hyperref}

\usepackage{datetime}

\newdateformat{mydate}{\THEDAY~\monthname[\THEMONTH]~\THEYEAR}

\usepackage{calc}

\usepackage{tikz}
\usetikzlibrary{matrix}
\usetikzlibrary{shapes.geometric, positioning, calc, patterns}

\definecolor{darkblue}{rgb}{0.0,0,0.7} 
 
\definecolor{darkred}{rgb}{0.7,0,0} 
\usepackage{hyperref}
\usepackage[all]{xy}
\usepackage[T1]{fontenc}

\usepackage{caption,lipsum}
\usepackage{graphicx}                  
\usepackage{pstricks,pst-plot,pst-text,pst-tree,pst-eps,pst-fill,pst-node,pst-math}
\usepackage{setspace}
\usepackage{multicol}
\usepackage{tikz-cd}
\usepackage{booktabs}

\usetikzlibrary{positioning}

\newcommand{\darkred}{\color{darkred}} 
\newcommand{\defn}[1]{\emph{\darkred #1}}

\usetikzlibrary{arrows.meta}

\def\M{{\mathsf{AM}}}

\usepackage{etex}

\newtheorem{theorem}{Theorem}[section]
\newtheorem{prop}[theorem]{Proposition}
\newtheorem{lemma}[theorem]{Lemma}
\newtheorem{cor}[theorem]{Corollary}
\newtheorem{conjecture}[theorem]{Conjecture}

\theoremstyle{definition}
\newtheorem{definition}[theorem]{Definition}
\newtheorem{rmq}[theorem]{Remark}
\newtheorem{exple}[theorem]{Example}

\newtheorem{question}[theorem]{Question}

\numberwithin{equation}{section}

\title[Absolute moves spaces and noncrossing partition posets]{Absolute moved spaces and noncrossing partition posets in arbitrary Coxeter groups}

\author{Thomas Gobet}
\address{Université Clermont Auvergne, LMBP, UMR 6620 (CNRS), Campus des Cézeaux, 3 place Vasarely, TSA 60026, CS 60026, 63178 Aubière cedex, France}

\begin{document}
	\maketitle
	
	\begin{abstract}
		The interval $[1,c]_T$ between the identity element and a Coxeter element $c$ in the absolute order on a Coxeter group $W$ is a generalization of the poset of noncrossing partitions arising when $W$ is the symmetric group. When $W$ is finite, this poset is always a lattice, and it is natural to associate to every element $w\in [1,c]_T$ its \textit{moved space} $\mathsf{Mov}(w)=\mathrm{Im}(w - \mathrm{Id}_V)$ in the geometric representation $V$ of $W$. It has dimension equal to the reflection length $\ell_T(w)$ of $w$, and gives a realization of $[1,c]_T$ inside the lattice of subspaces of $V$. It is an important tool in the study of $[1,c]_T$. When $W$ is infinite, the moved space of an element $w\in [1,c]_T$ no longer has dimension $\ell_T(w)$ in general, and distinct elements may have the same moved space. 
		
		We propose a replacement for the moved space of an element $w\in [1,c]_T$ in an arbitrary Coxeter group, that we call \textit{absolute moved space} of $w$. This subspace $\M(w)$ of $V$ always contains $\mathsf{Mov}(w)$ and has dimension equal to $\ell_T(w)$, and distinct elements have distinct absolute moved spaces. This allows us to derive several properties of noncrossing partition posets that hold in full generality, and to show that the natural map from $[1,c]_T$ to reflection subgroups of $W$, which to $w\in [1,c]_T$ associates the subgroup $P(w)$ generated by reflections lying below $w$ in the absolute order, is always injective. Among others, we also derive a new proof of the lattice property of $[1,c]_T$ when $W$ has rank three, and exhibit infinitely many new examples of infinite Coxeter groups of rank four and choices of Coxeter elements for which $[1,c]_T$ fails to be a lattice.    
	\end{abstract}
	
	\tableofcontents
	
	\section{Introduction}
	
	The lattice of noncrossing partitions of $\{1, 2, \dots, n\}$ admits a group-theoretic realization inside the symmetric group $\mathfrak{S}_n$, as the poset $[1,c]_T$ between the identity element and the $n$-cycle $c=(1 \ 2 \ \cdots \ n)$ in the absolute order~\cite{Biane}. This order is defined using the set $T$ of all transpositions. This allows one to generalize noncrossing partition posets to an arbitrary Coxeter group $W$, by taking for $c$ a Coxeter element and for $T$ the set of reflections: it is defined as the interval $[1,c]_T$ in the absolute order $\leq_T$ on $W$, which is defined by $$u \leq_T v \Leftrightarrow \ell_T(u) + \ell_T(u^{-1} v) = \ell_T(v),$$ where $\ell_T$ denotes the length function on $W$ with respect to $T$.  
	
	This generalization of the noncrossing partition lattice turned out to be of particular interest in the study of Artin groups, as initiated by Bessis~\cite{Dual} and Brady and Watt~\cite{Brady, BW_geom}. When $W$ is finite, all Coxeter elements are conjugate, hence different choices of Coxeter elements yield isomorphic posets since the set $T$ is stable by conjugation. Morever, it is always a lattice in this case. One can associate to such a poset a so-called \textit{interval group}, the dual Artin group $B_c^*$ which is isomorphic to the usual Artin group $B_W$ of $W$ in several cases, including when $W$ is finite or affine (and conjecturally in general). The lattice property guarantees that $B_c^*$ admits a so-called \textit{(quasi-)Garside structure}; this means that $B_c^*$ is the group of fractions of a monoid $M_c^*$, called \textit{dual braid monoid}, with enough divisibility properties to ensure that most combinatorial questions on $B_c^*$ can be solved: this includes the word and conjugacy problems, the absence of torsion, the determination of the center, and so on--see~\cite{Garside} for more on Garside theory. It is also central for the $K(\pi, 1)$-conjecture~\cite{Brady, BW_geom, PS, MS}.
	
	In the case where $W$ is finite, all the above questions can be solved also with a so-called \textit{classical} approach, taking for generating set the set $S$ of simple reflections instead of the set $T$. This indeed also yields a Garside structure. But when $W$ is infinite, the attached Artin group cannot be the group of fractions of the Artin monoid (because there are pairs of elements failing to have common multiples). One of the advantages of the "dual" approach is that the monoid $M_c^*$ and the group $B_c^*$ can be defined for arbitrary $W$ and choice $c\in W$ of Coxeter element, not necessarily finite, and even if asking for $B_c^*$ to be Garside is too much to expect, it is still expected that $B_c^*$ is the group of fractions of $M_c^*$. Unfortunately, in general, one may have several nonisomorphic noncrossing partition posets yielding several nonisomorphic monoids, and the lattice property fails in general. Outside spherical type, the dual Artin group $B_c^*$ was shown to be isomorphic to $B_W$ for Artin groups of affine type~\cite{Dig1, Dig2, MS, PS}, of universal type~\cite{Bessis_free} (i.e., free groups), for Artin-groups of rank three~\cite{DPS}, and recently for all Artin groups of extra-large type~\cite{Obrien}. In the affine types, the poset $[1,c]_T$ fails to be a lattice in several cases~\cite{McCammond}, while in universal~\cite{Bessis_free} and rank three cases~\cite{DPS, Gobet_max_dih} it is always a lattice. In other types, the situation is completely open. One of the main reasons for this is that the poset $[1,c]_T$ remains quite mysterious in general. 
	
	\medskip
	
	The goal of this paper is to introduce tools in order to initiate a study of the posets $[1,c]_T$ for general Coxeter groups. The main object we introduce is a generalization of the so-called \textit{moved space} $\mathsf{Mov}(w)$ of an element $w\in W$, that we call \textit{absolute moved space} and denote $\mathsf{AM}(w)$. It is a subspace of the geometric representation $V$ of $W$, and coincides with the moved space when $W$ is finite (but not in general). It satisfies $\dim(\mathsf{AM}(w)) = \ell_T(w)$. Moreover, if $x\leq_T c$ and $y\in W$ is the Kreweras complement of $x$, that is, if $xy = c$, then $\M(x) \oplus \M(y) = V = \M(c)$. 
	
	Let $x\in [1,c]_T = \{ w\in W \mid w \leq_T c\}$. Let $T(x) = \{ t\in T \mid t\leq_T x\}$. Let $P(x)$ denote the subgroup of $W$ generated by $T(x)$. It is a reflection subgroup of $W$, that is, a subgroup generated by reflections. It is parabolic when $W$ is finite, but not in general (see Section~\ref{prelim} below for more details). Let $\mathrm{Sub}(V)$ denote the poset of subspaces of $V$, partially ordered by inclusion. Let $\mathrm{RS}(W)$ denote the set of reflection subgroups of $W$, partially ordered by inclusion. Let $\mathcal{P}(T)$ denote the set of subsets of $T$, partially ordered by inclusion. The first main result is a generalization to all Coxeter groups and Coxeter elements $c$ of a result known to hold when $W$ is finite (which is a corollary of the main result of~\cite{BW_ortho}; see Theorem~\ref{thm_finite} below or~\cite[Section 2.4]{Gobet_sortable}). It claims that several maps from $[1,c]_T$ are injective and order-preserving: 
	
	\begin{theorem}[Theorems~\ref{thm_inj_subspaces}, ~\ref{poset_2},~\ref{poset_3}]\label{thm_mapinj}
	Let $(W,S)$ be a Coxeter system and $c\in W$ a Coxeter element. The three following maps are order-preserving and injective: 
	\begin{enumerate}
		\item The map $ \mathsf{AM}: [1,c]_T \longrightarrow \mathrm{Sub}(V), \ x \mapsto \mathsf{AM}(x)$,
		\item The map $P: [1, c]_T \longrightarrow \mathrm{RS}(W), \ x\mapsto P(x),$
		\item The map $T: [1,c]_T \longrightarrow \mathcal{P}(T), \ x \mapsto T(x).$
 	\end{enumerate} 
	\end{theorem}  
	
	When $W$ is finite, we have a stronger property: all these maps are isomorphisms of posets onto their images. This fails for the first map in general (see~Example~\ref{ex:fail_poset} below for a counterexample in rank four), and we conjecture it to hold for the last two (see Conjecture~\ref{conj_isom_3maps} below). It holds for the three maps for Coxeter groups of rank three: 
	
	\begin{prop}[{Corollary~\ref{isom_rank_three}}]\label{prop_3_maps}
		Let $(W,S)$ be a Coxeter system of rank three and $c\in W$ a Coxeter element. The three maps from Theorem~\ref{thm_mapinj} are isomorphisms of posets onto their images. 	
	\end{prop}	
	
	The following is another general property of noncrossing partition posets (see Proposition~\ref{prop_sets} below for a more precise statement): 
	
	\begin{prop}[{Proposition~\ref{prop_sets}}]\label{prop_set_refn}
	Let $(W,S)$ be a Coxeter system and $c\in W$ a Coxeter element. Let $x\in [1,c]_T$ and let $t_1 t_2 \cdots t_k$ be a $T$-reduced expression of $x$. If $y\in [1,c]_T$ admits a $T$-reduced expression whose letters are exactly the $t_i$'s, then $x=y$.   
	\end{prop}
	
	The next set of results concerns the lattice property of $[1,c]_T$. As explained above, when $W$ is infinite, the poset $[1,c]_T$ fails to be a lattice in general for some Coxeter groups and choices of Coxeter elements. 
	
	Recently, Delucchi, Paolini and Salvetti~\cite{DPS} proved that when $W$ is of rank $3$ and hyperbolic, the poset $[1,c]_T$ is a lattice for every choice of Coxeter element. Since the spherical and affine cases had been positively settled up before, this shows that $[1,c]_T$ is a lattice in every Coxeter group of rank three and for every choice of Coxeter element. An alternative proof which does not need to distinguish between finite, affine and hyperbolic types was later derived by the author from a more general statement, claiming that an interval $[u,v]_T=\{w\in W \mid u \leq_T w \leq_T v\}$ inside a Coxeter group $W$, with $u \leq_T v$ and $\ell_T(v) = \ell_T(u)+3$, is always a lattice~\cite{Gobet_max_dih}. Below, we derive yet another proof from Theorem~\ref{thm_mapinj}, using a dimension argument as in the proof from~\cite{DPS}, working for every Coxeter system of rank three. It is obtained from a statement holding in an arbitrary noncrossing partition poset $[1,c]_T$, claiming that bowties, which are obstructions to the lattice property, cannot occur in successive ranks (see Proposition~\ref{prop:bow_succ} below and the paragraph before for more details).   
	
	\begin{theorem}[{\cite[Theorem 4.2]{DPS}, \cite[Theorem 2.2]{Gobet_max_dih}, Corollary~\ref{cor_new_three}}]\label{thm_3}
	Let $(W,S)$ be a Coxeter system of rank $3$ and $c\in W$ a Coxeter element. Then $[1,c]_T$ is a lattice. 
	\end{theorem}        
	
	Outside finite, affine, rank three, and universal Coxeter systems, where we have a full classification of those Coxeter systems and choices of Coxeter elements for which $[1,c]_T$ is a lattice, the question is entirely open. 
	
	Our last result provides a new, infinite family of Coxeter systems of rank four and choices of Coxeter elements for which the lattice property of $[1,c]_T$ fails. As a corollary, if $(W,S)$ is a Coxeter system and $c\in W$ is a Coxeter element, such that $(W,S)$ contains one of these rank four Coxeter systems as a standard parabolic subgroup and the Coxeter element $c$ restricts to one for which the lattice property fails in this parabolic subsystem, then it also fails in $(W,S)$--see Corollary~\ref{cor_fail}. 
	
	\begin{theorem}[Theorem~\ref{thm_fail_lattice}]\label{thm_fail_latt}
	Let $(W,S)$ be a Coxeter system of rank $4$ with a diagram of the form \[\begin{tikzcd}
		s_1 && s_2 \\
		\\
		s_4 && s_3
		\arrow["{2k+1}", no head, from=1-1, to=1-3]
		\arrow["\ell", no head, from=1-3, to=3-3]
		\arrow["\ell", no head, from=3-1, to=1-1]
		\arrow["{2k' +1}", no head, from=3-3, to=3-1]
	\end{tikzcd}\]
	where $\ell \geq 3$, $k,k' \geq 1$. Let $c= s_3 s_1 s_2 s_4$. Then $[1,c]_T$ is not a lattice. 
	\end{theorem}  
	
	If $k=k'=1$ and $\ell=3$, then $W$ is of type $\widetilde{A_3}$, and the failure of the lattice property in this case was observed by Digne~\cite[Proposition 5.5]{Dig1}. Note that the diagrams above cover some hyperbolic Coxeter systems. 
	
	\bigskip
	
	The paper is organized as follows. In Section~\ref{prelim}, we collect some preliminary results on Coxeter groups, reflection subgroups, the absolute order, moved spaces, Coxeter elements. In particular we give in Theorem~\ref{thm_cox_proper} a short proof of the fact that Coxeter elements are essential. We also present a state of the art of some properties of $[1,c]_T$, such as the lattice property or, in the finite case, some relevant realizations of this poset. In some cases we give proofs for the sake of completeness. In Section~\ref{sec:abs}, we define absolute moved spaces and prove some basic properties. In Section~\ref{sec:inj}, we prove Theorem~\ref{thm_mapinj} and give properties of the three maps considered. In the last Section~\ref{sec:applications}, devoted to applications of absolute moved spaces in the study of the poset $[1,c]_T$, we prove Proposition~\ref{prop_set_refn} and Theorems~\ref{thm_3} and~\ref{thm_fail_latt}.  
	
	\section{Preliminaries and motivation}\label{prelim}
	
	\subsection{Coxeter groups and absolute order} Let $(W,S)$ be a Coxeter system. In this paper we shall always assume that $S$ is finite. For background on Coxeter groups we refer the reader to~\cite{Bourbaki, Humphreys, BB}. 
	Let $V$ denote the geometric representation of $W$, which is an $|S|$-dimensional $\mathbb{R}$-vector space endowed with a bilinear, symmetric form $B(-, -)$ called the \defn{Tits form}. For $w\in W$, let $\mathsf{Mov}(w)$ denote $\mathrm{Im}(w - \mathrm{Id}_V)$ and $V^w=\{v\in V \mid w(v)=v\}$. When $W$ is finite, the Tits form is a scalar product, and $V^w$ is the orthgonal subspace of $\mathsf{Mov}(w)$. Let $T:=\bigcup_{w\in W} w S w^{-1}$ be the set of reflections of $W$. We denote by $\Phi = \Phi^+ \sqcup \Phi^-$ the corresponding root system and by $t \leftrightarrow \alpha_t$ the $1$ to $1$ correspondence between reflections and positive roots. For $t\in T$ we sometimes denote its reflecting hyperplane by $H_t$ instead of $V^t$. 
	
	Given a Coxeter system $(W,S)$ and $I \subseteq S$, the pair $(W_I:= \langle I \rangle, I)$ is again a Coxeter system, called a \defn{standard parabolic subgroup} of $W$. A subgroup $P\subseteq W$ is a \defn{parabolic subgroup} if there are $I\subseteq S$ and $w\in W$ such that $P = w W_I w^{-1}$. Parabolic subgroups are stable by intersection (see for instance~\cite[Lemma 2.25]{AB}). A subgroup $W'$ of $W$ is a \defn{reflection subgroup} is there is $A\subseteq T$ such that $W' = \langle A \rangle$. Parabolic subgroups are examples of reflection subgroups. Every reflection subgroup has a canonical structure of Coxeter group $(W',S')$ with $S'\neq A$ in general, see~\cite{Deodhar, Dyer_subgroups}.
	
	A \defn{Coxeter element} in $(W,S)$ is a product of all the elements of $S$ in some order. We call a Coxeter element in $(W_I, I)$ for $I\subseteq S$ a \defn{standard parabolic Coxeter element}.  
	
	The~\defn{reflection length} $\ell_T(w)$ of an element $w\in W$ is the smallest nonnegative integer $k$ for which there exist reflections $t_1, t_2, \dots, t_k\in T$ such that $w=t_1 t_2 \cdots t_k$. Replacing $T$ by $S$ in the previous definition yields the (classical) length function $\ell_S$ on $w$. If $t_1, t_2, \dots, t_k\in T$ and $w= t_1 t_2 \cdots t_k$ with $k=\ell_T(w)$, we say that the word $t_1 t_2 \cdots t_k$ is a \defn{$T$-reduced expression} of $w$ (and an $S$-reduced expression if $T$ is replaced by $S$ everywhere). The \defn{absolute order} $\leq_T$ on $W$ is then defined as $ u \leq_T v$ if and only if $\ell_T(u) + \ell_T(u^{-1} v)= \ell_T(v)$, that is, there is a $T$-reduced expression of $v$ having a $T$-reduced expression of $u$ as a prefix (equivalently as a suffix, since $T$ is stable by conjugation). 
	
	It is natural to seek for interpretations of $\ell_T(w)$ in terms of $V$ or $\Phi$ and algorithms to determine it. In the case where $W$ is finite one has the following geometric interpretation, due to Carter: 
	
	\begin{theorem}[Carter's Lemma, 1972, {\cite[Lemmas 1 - 3]{Carter}}]\label{thm_carter}
		Let $(W,S)$ be a finite Coxeter group. 
		\begin{enumerate}
			\item For all $w\in W$, we have $\ell_T(w)= \dim(V) - \dim(V^w) = \dim(\mathsf{Mov}(w))$. 
			\item For all $t\in T$, $w\in W$, we have $t \leq_T w \Leftrightarrow V^w \subseteq V^t \Leftrightarrow \mathsf{Mov}(t) \subseteq \mathsf{Mov}(w).$
			\item Given $w\in W$ and $t_1, t_2, \dots, t_k\in T$ such that $w= t_1 t_2 \cdots t_k$, we have $\ell_T(w)=k$ if and only if the roots $\{\alpha_{t_i}\}_{i=1}^k$ are linearly independent. 
		\end{enumerate}
	\end{theorem}
	
	This theorem does not generalize to arbitrary $W$, where typically the reflection length can be unbounded in general~\cite{Dus}. In affine types one can derive a formula analogous to the one in point $(1)$ (see~\cite{BLMC}). For arbitrary Coxeter systems, we have the following, which allows one to calculate $\ell_T(w)$ for any element $w$ of a Coxeter group: 
	
	\begin{theorem}[Dyer, 2001, {\cite[Theorem 1.1]{Dyer_refn}}]\label{thm_dyer_ref_length} 
		Let $(W,S)$ be an arbitrary Coxeter system. Let $w\in W$. Let $s_1 s_2 \cdots s_k$ be any $S$-reduced expression of $w$. Then $\ell_T(w)$ is equal to the minimal number of letters to delete in the word $s_1 s_2 \cdots s_k$ to get a word representing the identity.  
	\end{theorem}

	\subsection{Moved spaces of elements of finite Coxeter groups}
	When $W$ is finite, one can wonder whether point $(2)$ of Theorem~\ref{thm_carter} stays true if $t$ is replaced by an arbitrary element $u\in W$. Unfortunately, the equivalence does not hold in general--only the direct implication holds. For equivalence one needs a further assumption; we have the following result, which is a corollary of the main result from~\cite{BW_ortho} (see also~\cite[Theorem 2.4.9]{Arm}):  
	
	\begin{prop}\label{prop_bw}
	Let $(W,S)$ be a finite Coxeter group. 
	\begin{enumerate}
		\item If $v, w\in W$, then $ v \leq_T w \Rightarrow V^w \subseteq V^v$ (equivalently $\mathsf{Mov}(v) \subseteq \mathsf{Mov}(w)$)
		\item Let $v, w\in W$. Assume that there is $u\in W$ such that $v, w \leq_T u$. Then $ \Leftrightarrow v \leq_T w \Leftrightarrow V^w \subseteq V^v$ (equivalently $\mathsf{Mov}(v)\subseteq \mathsf{Mov}(w)$). 
	\end{enumerate}
	\end{prop}
	
	\begin{proof}
	\begin{enumerate}
		\item Assume that $v \leq_T w$, and let $t_1 t_2 \cdots t_k$ be a $T$-reduced expression of $w$ such that $t_1 t_2 \cdots t_i$ is a $T$-reduced expression of $v$ for some $i\in \{0, 1, \dots, k\}$. It follows from Theorem~\ref{thm_carter} that $$V^w  = \bigcap_{j=1}^k H_{t_j} \subseteq \bigcap_{j=1}^i H_{t_j} = V^v.$$
		\item Assume that there is $u\in W$ such that $v, w \leq_T u$. The direct implication is given by point (1). Assume that $V^w \subseteq V^v$. Taking orthogonal spaces yields $\mathsf{Mov}(v) \subseteq \mathsf{Mov}(w)$. Since $W$ is finite, one can view $W$ inside the orthogonal group $\mathrm{O}(V)$ with respect to the Tits form $B(-,-)$, which is a scalar product in this case. For $x, y\in \mathrm{O}(V)$, set $x \leq y$ if $$\dim(\mathsf{Mov}(x)) + \dim(\mathsf{Mov}(x^{-1} y)) = \dim(\mathsf{Mov}(y)).$$
		By~\cite{BW_ortho}, this defines a partial order on $\mathrm{O}(V)$. By Theorem~\ref{thm_carter}, we have $v,w \in \{ x \in \mathrm{O}(V) \ \mid \ 1 \leq x \leq u\}$. Now by~\cite[Theorem 2]{BW_ortho}, this interval is a lattice, isomorphic to a lattice of subspaces of $V$ via $x \mapsto \mathsf{Mov}(x)$. Since $\mathsf{Mov}(v) \subseteq \mathsf{Mov}(w)$, this implies that $v \leq w$. But since $v, w\in W$, we also have $v^{-1} w \in W$, hence by Theorem~\ref{thm_carter} the equality $\dim(\mathsf{Mov}(v)) + \dim(\mathsf{Mov}(v^{-1} w)) = \dim(\mathsf{Mov}(w))$ can be rewritten $\ell_T(v) + \ell_T(v^{-1} w) = \ell_T(w)$, yielding $v \leq_T w$.    
	\end{enumerate}
	\end{proof}
	
	\subsection{Various realizations of the poset $[1,c]_T$ for finite Coxeter groups}\label{various} Let $c$ be a Coxeter element in a Coxeter system $(W, S)$. It follows from Theorem~\ref{thm_dyer_ref_length} that $\ell_T(c) = \ell_S(c) = |S|$. 
	
	Now assume that $W$ is finite. In this case, the fact that $\ell_T(c) = |S|$ can also be established by point (3) of Theorem~\ref{thm_carter}. By the same Theorem, we have that every Coxeter element is a maximal element for the absolute order $\leq_T$, since the reflection length cannot exceed $\dim(V) = |S|$ when $W$ is finite. In particular, unless $W$ is a direct product of irreducible components of type $A_1$, the poset $(W, \leq_T)$ is never a lattice. But the question can be asked if one restricts to the order ideal of a maximal element, for instance a Coxeter element $c$. Let $[1,c]_T = \{ x \in W \mid x \leq_T c\}$. By point (2) of Proposition~\ref{prop_bw}, the map $x \mapsto \mathsf{Mov}(x)$ provides a realization of the poset $[1,c]_T$ inside the lattice of subspaces of $V$. 
	
	The moved space $\mathsf{Mov}(x)$ of an element $x\in W$ is thus a very useful tool in the study of the poset $[1,c]_T$ and its elements when $W$ is finite. Unfortunately, for infinite $W$, it is no longer true that $\dim(\mathsf{Mov}(x)) = \ell_T(x)$ in general. It is desirable to find a replacement, which we shall do by introducing absolute moved spaces in the next section.
	
	There are at least two other realizations of the poset $[1,c]_T$, one inside the poset of parabolic subgroups of $(W,S)$ partially ordered by inclusion, the other one inside $\mathcal{P}(T)$, also partially ordered by inclusion. We recall them now (see also~\cite[Section 2.4]{Gobet_sortable}).
	
	Let $x\in W$ and let $P(x) = \langle t \mid t\leq_T x \rangle$. By Theorem~\ref{thm_carter}, for $t\in T$ we have $t \leq_T x$ if and only if $V^x \subseteq V^t = H_t$, if and only if $t \in C_W(V^x)$. This last group is a parabolic subgroup of $W$ (centralizers of subsets of $V$ coincide with parabolic subgroups of $W$), and since $P(x)$ is by definition a reflection subgroup, they are both generated by reflections. The equivalence $t \leq_T x \Leftrightarrow t \in C_W(V^x)$ (with $t\in T$) implies that both groups are equal, i.e., we have $P(x) = C_W(V^x)$. We thus have that $P(x)$ is the parabolic closure of $x$, that is, the smallest parabolic subgroup of $W$ containing $x$. We also have that the set $P(x) \cap T$ of reflections of $P(x)$ is the same as the set of reflections of $C_W(V^x)$, and is given by $T(x):=\{ t\in T \mid t \leq_T x\}$.
	
	Let $x, y \in W$. We have \begin{equation}\label{eq_px} P(x) \subseteq P(y) \Leftrightarrow V^y \subseteq V^x.\end{equation}
	Indeed, since $P(w) = C_W(V^w)$, the reverse implication is immediate. For the direct implication, let $t_1 t_2 \cdots t_k$ be a $T$-reduced expression of $x$. Then $t_i \in P(x) \subseteq P(y)$ for all $i$. We thus have $V^y \subseteq H_{t_i}$ for all $i$. By Theorem~\ref{thm_carter} we have $V^x = \bigcap_{i=1}^k H_{t_i}$. This yields $V^y \subseteq V^x$. 	
	 
	Now assume that $x, y\in [1,c]_T$. Combining the above properties we have $$ x \leq_T y \Leftrightarrow \mathsf{Mov}(x) \subseteq \mathsf{Mov}(y) \Leftrightarrow V^y \subseteq V^x \Leftrightarrow P(x) \subseteq P(y) \Leftrightarrow T(x) \subseteq T(y),$$
	 where the first two equivalences are given by point (2) of Proposition~\ref{prop_bw}, the third one is~\eqref{eq_px}, and the last one is immediate since $T(w) = P(w) \cap T$ for all $w\in W$. 
	 
	 Let $\mathrm{PS}(W)$ denote the set of parabolic subgroups of $W$, partially ordered by inclusion. To summarize we have the following realizations of the poset $[1,c]_T$ when $W$ is finite: 
	 
	 \begin{theorem}\label{thm_finite}
	 Let $(W,S)$ be a finite Coxeter group and $c\in W$ a Coxeter element. The three following maps are isomorphisms of posets onto their images: 
	 \begin{enumerate}
	 	\item The map $\mathsf{Mov}: [1,c]_T \longrightarrow \mathrm{Sub}(V), \ x \mapsto \mathsf{Mov}(x)$,
	 	\item The map $P: [1,c]_T \longrightarrow \mathrm{PS}(W), \ x \mapsto P(x)$,
	 	\item The map $T: [1,c]_T \longrightarrow \mathcal{P}(T), \ x \mapsto T(x)$.
	 \end{enumerate}
	 \end{theorem}
	  
	The various sets attached to $x\in [1,c]_T$ and the corresponding maps are useful in several contexts; we give a (non-exhaustive) list of situations where they occur. When $W$ is finite, the group $P(x)$ has rank $\ell_T(x)$, and $x$ is a Coxeter element in $P(x)$. This allows one to argue by induction when showing properties of $[1,c]_T$ (see~\cite{Dual}). The map $x \mapsto T(x)$ is instrumental in Brady and Watt's uniform proof of the lattice property of $[1,c]_T$ when $W$ is finite (see~\cite[Section 2]{BW_lattice} and~\cite[Lemmas 2.5 and 2.6]{BW_geom}). Also, in Coxeter-Catalan combinatorics, when showing that a map from $[1,c]_T$ to another set is a bijection, it is common to use the above maps to show that such maps are injective~\cite{Reading_sortable, RS, Abouyassin}. 
	
	\subsection{Properties of Coxeter elements}
	
	Coxeter elements enjoy remarkable properties. In this subsection we list some of them. 
	
	There is a natural action of the $n$-strand braid group on $T$-reduced expressions of an element $w\in W$. Writing a $T$-reduced expression $t_1 t_2 \cdots t_k$, with $k= \ell_T(w)$, as a $k$-tuple of reflections $(t_1, t_2, \dots, t_k)$, the standard Artin generator $\sigma_i$ acts as $$\sigma_i \cdot (t_1, t_2, \dots, t_{i-1}, t_i, t_{i+1}, t_{i+2}, \dots, t_k) =(t_1, t_2, \dots, t_{i-1}, t_{i+1}, t_{i+1}t_i t_{i+1}, t_{i+2}, \dots, t_k).$$
	An important result is the following. For finite Coxeter group, the first published proof is due to Bessis~\cite{Dual}, but a proof appears in a letter from Deligne to Looijenga from March 1974, where Deligne attributes the proof to Tits and Zagier. The proof for arbitrary Coxeter groups is due to Igusa and Schiffler~\cite{IS}. See also~\cite{BDSW} for another proof with interesting consequences. 
	
	\begin{theorem}[{\cite[Theorem 1.4]{IS}}]\label{thm_hurw_trans}
	Let $(W,S)$ be a Coxeter system and $c\in W$ a Coxeter element. The Hurwitz action on $T$-reduced expressions of $c$ is transitive. 
	\end{theorem}
	
	 In particular, every two $T$-reduced expressions of $c$ can be related by a sequence of local relations called \textit{dual braid relations}, which are of the form $ab = bc$. This is a kind of dual analogue (for Coxeter elements) of Matsumoto's property in the classical setting (i.e., with the generating set $S$). Nevertheless, for arbitrary elements the Hurwitz action is not transitive in general. 
	 
	 This result has several consequences. For example, for every $T$-reduced expression $t_1 t_2 \cdots t_n$ of $c$, we have $W = \langle t_1, t_2, \dots, t_n \rangle$. For finite Coxeter groups, the transitivity of the Hurwitz action is an important tool in Bessis' construction of dual braid monoids~\cite{Dual}. For arbitrary Coxeter groups, it allows one to show that Artin groups surject onto dual Artin groups (see~\cite[Corollary~{2.10}~(3)]{Obrien} for a proof of this fact; it is conjectured that both groups are isomorphic through this quotient map).

	We point out the following property, which we shall not need, but which is remarkable. Since every element $s\in S$ satisfies $s \leq_T c$, we have $P(c) = W$. It was shown by Paris~\cite[Theorem 3.1]{Paris_irred} that Coxeter elements are essential, that is, they do not lie in a proper parabolic subgroup of $W$. In fact, one can show more generally that a Coxeter element can never lie in a proper \textit{reflection} subgroup of $W$, as a consequence of Dyer's work on reflection subgroups~\cite{Dyer_subgroups} of Coxeter systems:

	\begin{theorem}\label{thm_cox_proper}
	Let $(W,S)$ be a Coxeter system and $c\in W$ a Coxeter element. Let $W'\subsetneq W$ be a proper reflection subgroup of $W$. Then $c\notin W'$. 
	\end{theorem}
	
	\begin{proof}
	Assume that $W'$ is a reflection subgroup of $W$ such that $c\in W'$. Let $S'$ denote the set of canonical generators of $W'$ as a Coxeter group, which is a subset of $T$. By~\cite[Theorem 3.3]{Dyer_subgroups}, we have $\ell_{S'}(c) = | N(c) \cap W' |$. Since $S' \subseteq T$, we have $$|S| =\ell_T(c) \leq \ell_{S'}(c) = | N(c) \cap W' | \leq |N(c)| = |S|.$$ We thus have equalities everywhere, hence  $N(c) \subseteq W'$. But $\langle N(c) \rangle = W$, hence $W'=W$. 
	\end{proof}

	\subsection{Arbitrary Coxeter groups} When $W$ is infinite, the various objects considered above may have very different properties, leading to serious obstructions if one tries to generalize Theorem~\ref{thm_finite}: 
	\begin{itemize}
		\item While one can still consider $\mathsf{Mov}(x)$, it is not well-behaved in general: we do not have $\dim(\mathsf{Mov}(x)) = \ell_T(x)$. It can happen that $u \leq_T v$ but $\mathsf{Mov}(u) \not\subseteq \mathsf{Mov}(v)$ (see Example~\ref{ex_a1_tilde} below). In addition, we do not have a direct sum decomposition $V = \mathsf{Mov}(u) \oplus V^u$ in general, \item The results from~\cite{BW_ortho} from which the equivalence of the second point of Proposition~\ref{prop_bw} is derived do not hold since they require the Tits form to be anisotropic, which does not hold outside the finite case,
		\item The subgroup $P(x)$ is not parabolic in general (see Example~\ref{ex_nonparab} below, which is borrowed from~\cite[Example 5.7]{HK}), 
		\item It is not true that $P(x) \cap T = T(x)$ in general: already taking $x=c$, one has $P(x) = W$, and when $W$ is infinite, there are in general reflections in $W$ which are not in $[1,c]_T$.  
	\end{itemize}
	
	\begin{exple}\label{ex_a1_tilde}
	Let $W$ be of type $\widetilde{A_1}$ with simple system $S=\{s,t\}$. Let $c= st$. Then $s \leq_T c$, $\ell_T(s)=1$, $\ell_T(c)= 2$, but $\mathsf{Mov}(s)= \mathbb{R} \alpha_s \not\subseteq\mathsf{Mov}(c) = \mathbb{R} \delta$ with $\delta = \alpha_s + \alpha_t$. Moreover, we have $V^c = \mathbb{R} \delta = \mathsf{Mov}(c)$.  
	\end{exple}
	
	\begin{exple}\label{ex_nonparab}
	Let $W$ be of type $\widetilde{A_2}$ with simple system $S=\{s,t,u\}$. Let $c= stu$. Since $c=stu = s (tut) t$, we have $x := s (tut) \leq_T c$. In this case $P(x)$ is generated by the reflections in any $T$-reduced factorization of $x$ (see Remark~\ref{hur_trans_pref} below; in general for elements of $[1,c]_T$ in an arbitrary Coxeter group this is only conjectural, see Conjecture~\ref{conj_hur_trans_prefix}), hence $P(x) = \langle s, tut \rangle$, which is an infinite dihedral group, while every proper parabolic subgroup of $W$ is finite. We thus have that $P(x)$ is not parabolic. 	
	\end{exple}
	
	It is nevertheless natural to wonder whether for $x, y\in [1,c]_T$, we still have $$x \leq_T y \Leftrightarrow P(x) \subseteq P(y) \Leftrightarrow T(x) \subseteq T(y)$$ or not, and to look for a replacement for $\mathsf{Mov}(x)$. As we shall see in the next section, there is a natural generalization of $\mathsf{Mov}(x)$ to arbitrary Coxeter groups given by the \textit{absolute moved space} $\M(x)$ of $x$, which allows a version of Theorem~\ref{thm_finite} for arbitrary Coxeter groups, given by Theorem~\ref{thm_mapinj}. We cannot get a result which is as strong as in the finite case, since we only get injective, order-preserving maps instead of maps which are isomorphisms of posets onto their images. We show that $ x \mapsto \M(x)$ is not an isomorphism of posets onto its image in general (see Example~\ref{ex:fail_poset}), but we conjecture that the other two are. 
	
	Note that some of these maps have been shown to be injective on $c$-sortable elements by Reading and Speyer~\cite[Theorem 8.9]{RS}, a remarkable subset of $W$ which is in bijection with $[1,c]_T$ when $W$ is finite, but only injects into $[1,c]_T$ when $W$ is infinite: for the analogue of map $(2)$, the subgroups caught in the image are always parabolic, and finite.  
	
	\subsection{Lattice property} Another property of $[1,c]_T$ that holds when $W$ is finite but not in general is the lattice property. This property is crucial in the study of dual Artin groups, as it provides a (quasi-)Garside structure on the attached interval group. For finite Coxeter groups, particular cases were treated by Brady~\cite{Brady} and Brady and Watt~\cite{BW_geom}, and a complete proof in all types was first given by Bessis on a case-by-case basis~\cite{Dual}. Brady and Watt later gave a uniform proof~\cite{BW_lattice}, where the set $T(x)$ plays a key role. Reading later gave another completely different, uniform proof~\cite{Reading_shards}. The lattice property also holds for universal Coxeter groups~\cite{Bessis_free}, for some of the affine types~\cite{Dig1, Dig2, McCammond, BR}, and for Coxeter systems of rank three~\cite{DPS, Gobet_max_dih}. Note that, when $W$ is finite, all Coxeter elements are conjugate and hence, since $T$ is stable by conjugation, the poset $[1,c]_T$ is independent of the choice of Coxeter element. Typically in type $\widetilde{A}_n$, the lattice property holds for some choices of Coxeter elements, but fails for others. We summarize the known results in Table~\ref{tab:treillis}.

	\begin{table}[h]
		\centering
		\begin{tabular}{cccc}
			\toprule
			$W$ & $c$ & Is $[1,c]_T$ a lattice? & References \\
			 \specialrule{1.2pt}{0pt}{0pt}
			Finite & arbitrary & yes & \cite{Brady, BW_geom, Dual, BW_lattice, Reading_shards} \\
			Universal & arbitrary & yes & \cite{Bessis_free} \\
			$\widetilde{A}_n$ & linear & yes & \cite{Dig1} \\
			$\widetilde{A}_n$ & not linear & no & \cite{Dig1} \\
			$\widetilde{C}_n$ & arbitrary & yes & \cite{Dig2, BR} \\
			$\widetilde{G}_2$ & arbitrary & yes & \cite{McCammond, DPS, Gobet_max_dih} \\ 
			Remaining affine types & arbitrary & no & \cite{McCammond} \\
			Rank three & arbitrary & yes & \cite{DPS, Gobet_max_dih}\\
			\bottomrule
		\end{tabular}
		\caption{Cases where the question whether $[1,c]_T$ is a lattice or not is settled.}
		\label{tab:treillis}
	\end{table} 
	
	In Section~\ref{sec:applications} below, we give a new proof of the lattice property of $[1,c]_T$ in rank three, as a corollary of Theorem~\ref{thm_mapinj}. Similarly to the proof of~\cite{DPS}, it uses the fact that two $2$-dimensional distinct subspaces of a $3$-dimensional vector space meet in a line. We also use absolute moved spaces to add a line to Table~\ref{tab:treillis}, by exhibiting an infinite family of Coxeter systems of rank four and choices of Coxeter element for which $[1,c]_T$ fails to be a lattice.     
	
	\section{Absolute moved spaces}\label{sec:abs}
	
	The aim of this section is to define absolute moved spaces of elements from $[1,c]_T$ and give some of their properties. We will denote by $(W,S)$ an arbitrary Coxeter system and $c\in W$ a Coxeter element in $W$. 
	
	\subsection{Euler form}
	
	We choose an $S$-reduced expression $s_1 s_2 \cdots s_n$ of $c$. The \defn{Euler form} $\varphi_c$ is the bilinear form $\varphi_c$ on $V$ defined by \[
	\varphi_c(\alpha_{s_i}, \alpha_{s_j}) =\left\{
	\begin{array}{ll}
		0 &  \text{if~}i < j, \\
		1 &  \text{if~}i=j, \\
		2 B(\alpha_{s_i}, \alpha_{s_j}) &  \text{if~}i > j.
	\end{array}
	\right.
	\]
	
	Note that any two $S$-reduced expressions of $c$ are related only by commutation of adjacent letters. It follows that the above form is independent of the chosen $S$-reduced expression for $c$ since whenever $st=ts$, we have $B(\alpha_s, \alpha_t) = 0$. It only depends on $c$.  
	
	This form appears (possibly with slightly different conventions) in several contexts involving noncrossing partitions, like representations of quivers~(see for instance~\cite{HK, IT}), or in the theory of $c$-sortable elements (see~\cite{RS}). 
	
	\medskip
	
	Recall that $\ell_S(c) = \ell_T(c) = |S|$. 
	
	\begin{lemma}\label{lem_form_hurwitz}
		Let $(t_1, t_2, \dots, t_n)$ be a $T$-reduced expression of $c$. Then 
		\[\varphi_c(\alpha_{t_i}, \alpha_{t_j})=\left\{\begin{array}{ll}
			0 &  \text{if~}i < j, \\
			1 &  \text{if~}i=j, \\
			2 B(\alpha_{t_i}, \alpha_{t_j}) &  \text{if~}i > j.
		\end{array}
		\right.
		\]
	\end{lemma}
	
	\begin{proof}
		Assume that the formula holds for one $T$-reduced expression $(t_1, t_2, \dots, t_n)$, and let $(t_1', t_2', \dots, t_n')$ denote  either $\sigma_i \cdot (t_1, t_2, \dots, t_n)$ or $\sigma_i^{-1} \cdot (t_1, t_2, \dots, t_n)$. Since the Hurwitz action is transitive on $T$-reduced expressions of $c$ and the formula holds by definition of $\varphi_c$ for the $T$-reduced expression $(s_1, s_2, \dots, s_n)$, it suffices to show that the claimed formula holds for $(t_1', t_2', \dots, t_n')$.

		We treat the case where $(t_1', t_2', \dots, t_n')=\sigma_i \cdot (t_1, t_2, \dots, t_n)$, the other computation is similar. We have $$(t_1', t_2', \dots, t_n')=(t_1, t_2, \dots, t_{i-1}, t_{i+1}, t_{i+1} t_i t_{i+1}, t_{i+2}, \dots, t_n).$$
		Using the fact that $\alpha_{t_{i+1} t_i t_{i+1}}$ is a linear combination of $\alpha_{t_i}$ and $\alpha_{t_{i+1}}$, we get that $\varphi_c(\alpha_{t_j'}, \alpha_{t_k'})=0$ for $j < k$ except possibly for $(j,k)=(i,i+1)$. In this case we have 
		\begin{align*}
			\varphi_c(\alpha_{t_{i+1}}, \alpha_{t_{i+1} t_i t_{i+1}}) &=\varphi_c(\alpha_{t_{i+1}}, \pm t_{i+1}(\alpha_{t_{i}})) = \pm \varphi_c(\alpha_{t_{i+1}}, \alpha_{t_i} - 2 B(\alpha_{t_{i+1}}, \alpha_{t_i}) \alpha_{t_{i+1}}) \\ & = \pm (\varphi_c(\alpha_{t_{i+1}}, \alpha_{t_i}) - 2 B(\alpha_{t_{i+1}}, \alpha_{t_i}))= 0.
		\end{align*}
		We have $\varphi_c(\alpha_{t_j'}, \alpha_{t_j'}) = 1$ for all $j$ except possibly when $j=i+1$. In this case we have 
		\begin{align*}
			\varphi_c(\alpha_{t_{i+1} t_i t_{i+1}}, \alpha_{t_{i+1} t_i t_{i+1}})& = \varphi_c( \pm t_{i+1}(\alpha_{t_i}), \pm t_{i+1}(\alpha_{t_i})) \\ & = \varphi_c( \alpha_{t_i} - 2 B(\alpha_{t_{i+1}}, \alpha_{t_i}) \alpha_{t_{i+1}},  \alpha_{t_i} - 2 B(\alpha_{t_{i+1}}, \alpha_{t_i}) \alpha_{t_{i+1}})\\ &= 1 + 4 B(\alpha_{t_{i+1}}, \alpha_{t_i})^2 - 4 B(\alpha_{t_{i+1}}, \alpha_{t_i})^2 = 1.
		\end{align*}
		Finally, we have $\varphi_c(\alpha_{t_j'}, \alpha_{t_k'}) = 2 B(\alpha_{t_j'}, \alpha_{t_k'})$ for $j > k$ except possibly for $i+1 \in \{j, k\}$. We check the remaining cases. First assume that $ j = i+1$ and $k=i$. If $t_{i+1}(\alpha_{t_i}) \in \Phi^+$, then using that $B(-,-)$ is $W$-invariant we have 
		\begin{align*}
			\varphi_c( \alpha_{t_{i+1} t_i t_{i+1}}, \alpha_{t_{i+1}}) & = \varphi_c(t_{i+1}(\alpha_{t_i}), \alpha_{t_{i+1}}) = \varphi_c(\alpha_{t_i} - 2 B(\alpha_{t_{i+1}}, \alpha_{t_i}) \alpha_{t_{i+1}}, \alpha_{t_{i+1}}) \\ &= - 2 B(\alpha_{t_{i+1}}, \alpha_{t_i}) = 2 B(t_{i+1}(\alpha_{t_{i+1}}), \alpha_{t_i}) = 2 B(\alpha_{t_{i+1}}, t_{i+1}(\alpha_{t_i})) \\ &= 2 B(\alpha_{t_{i+1}}, \alpha_{t_{i+1} t_i t_{i+1}}).
		\end{align*} 
		If $t_{i+1}(\alpha_{t_i}) \in \Phi^-$ then 
		\begin{align*}
			\varphi_c( \alpha_{t_{i+1} t_i t_{i+1}}, \alpha_{t_{i+1}}) & = \varphi_c(-t_{i+1}(\alpha_{t_i}), \alpha_{t_{i+1}}) = \varphi_c(-\alpha_{t_i} + 2 B(\alpha_{t_{i+1}}, \alpha_{t_i}) \alpha_{t_{i+1}}, \alpha_{t_{i+1}}) \\ &= 2 B(\alpha_{t_{i+1}}, \alpha_{t_i}) = -2 B(t_{i+1}(\alpha_{t_{i+1}}), \alpha_{t_i}) = 2 B(\alpha_{t_{i+1}}, -t_{i+1}(\alpha_{t_i})) \\ &= 2 B(\alpha_{t_{i+1}}, \alpha_{t_{i+1} t_i t_{i+1}}).
		\end{align*}
		Now assume that $j= i+1$ and $k < i$. If $t_{i+1}(\alpha_{t_i})\in \Phi^+$ then 
		\begin{align*}
			\varphi_c( \alpha_{t_{i+1} t_i t_{i+1}}, \alpha_{t_{k}}) & = \varphi_c(t_{i+1}(\alpha_{t_i}), \alpha_{t_{k}}) = \varphi_c(\alpha_{t_i} - 2 B(\alpha_{t_{i+1}}, \alpha_{t_i}) \alpha_{t_{i+1}}, \alpha_{t_{k}}) \\ &= 2 B(\alpha_{t_i}, \alpha_{t_k}) - 4 B(\alpha_{t_{i+1}}, \alpha_{t_i}) B(\alpha_{t_{i+1}}, \alpha_{t_k}) \\ &= 2B(\alpha_{t_i} - 2 B(\alpha_{t_{i+1}}, \alpha_{t_i}) \alpha_{t_{i+1}}, \alpha_{t_{k}}) = 2 B( t_{i+1}(\alpha_{t_i}), \alpha_{t_k})\\ & = 2 B( \alpha_{t_{i+1} t_i t_{i+1}}, \alpha_{t_k}),
		\end{align*} and similarly we get the result if $t_{i+1}(\alpha_{t_i})\in \Phi^-$. 
		
		Finally, assume that $k = i+1$, hence $j > i+1$ and $t_j'=t_j$. If $t_{i+1}(\alpha_{t_i})\in \Phi^+$ then  
		\begin{align*}
			\varphi_c( \alpha_{t_{j}}, \alpha_{t_{i+1} t_i t_{i+1}}) & = \varphi_c(\alpha_{t_j}, t_{i+1}(\alpha_{t_i})) = \varphi_c(\alpha_{t_j}, \alpha_{t_i} - 2 B(\alpha_{t_{i+1}}, \alpha_{t_i}) \alpha_{t_{i+1}}) \\ &= 2 B(\alpha_{t_j}, \alpha_{t_i}) - 4 B(\alpha_{t_j}, \alpha_{t_{i+1}}) B(\alpha_{t_{i+1}}, \alpha_{t_i}) \\ &= 2 B(\alpha_{t_j}, \alpha_{t_i} - 2 B(\alpha_{t_{i+1}}, \alpha_{t_i}) \alpha_{t_{i+1}}) = 2 B( \alpha_{t_j}, t_{i+1}(\alpha_{t_i})) \\ &= 2B(\alpha_{t_j}, \alpha_{t_{i+1} t_i t_{i+1}}), 
		\end{align*}
		and similarly we get the result if $t_{i+1}(\alpha_{t_i}) \in \Phi^-$. 
	\end{proof}

	\begin{exple}
	While Lemma~\ref{lem_form_hurwitz} will be instrumental to show properties of absolute moved spaces, it can also be useful for instance to check that a product $tt'$ of two reflections is not in $[1,c]_T$. As an example, let $W$ be of type $\widetilde{A}_2$ with $S= \{s_1,s_2,s_3\}$ and $c=s_1 s_2 s_3$. Let $t = s_2$ and $t' = s_1 s_3 s_1$. We have $$\varphi_c(\alpha_t, \alpha_{t'}) = \varphi_c( \alpha_2, \alpha_1 + \alpha_3) = 2 B(\alpha_2, \alpha_1) = -1 \neq 0,$$ hence $tt' \not\leq_T c$. 
	\end{exple}

\begin{definition}
We define a skew-symmetric form $\omega_c$ on $V$ by setting $\omega_c( \alpha_{s_i}, \alpha_{s_j}) = \varphi_c( \alpha_{s_i}, \alpha_{s_j}) - \varphi_c( \alpha_{s_j}, \alpha_{s_i})$. 
\end{definition}
	
This form appears in~\cite{Speyer, RS}. As a corollary of Lemma~\ref{lem_form_hurwitz} we have the following result. 
	
\begin{lemma}\label{lem_omega_2}
Let $t, t'\in T$ such that $tt' \leq_T c$. Then $$\omega_c(\alpha_t, \alpha_{t'}) = - 2 B(\alpha_t, \alpha_{t'}).$$
\end{lemma}

\begin{proof}
If $tt' \leq_T c$, there is a $T$-reduced expression of $c$ of the form $(t, t', t_3, \dots, t_n)$. By Lemma~\ref{lem_form_hurwitz} we have $\varphi_c(\alpha_t, \alpha_{t'}) = 0$ and $\varphi_c(\alpha_{t'}, \alpha_t) = 2 B(\alpha_t, \alpha_{t'})$, yielding $\omega_c(\alpha_t, \alpha_{t'}) = -2 B(\alpha_t, \alpha_{t'})$. 
\end{proof}
	
	\subsection{Definition and properties of absolute moved spaces}

	\begin{definition}[Absolute moved space]\label{def:wms}
		Let $(W,S)$ be a Coxeter system and $c\in W$ a Coxeter element. Let $x\in [1, c]_T$. The \defn{absolute moved space} of $x$ is the subspace $\M(x)$ of $V$ defined by $$\M(x) = \sum_{t\in T, \ t\leq_T x} \mathbb{R} \alpha_t.$$
	\end{definition}
	
	Note that the definition of $\M(x)$ does not have any reference to $c$, and we could give such a definition for an arbitrary element $x\in W$. Nevertheless, it will not have the same properties in general as those which will be required for our purposes. See also Remark~\ref{rmq:genmov} below. 
	
	\begin{rmq}[Traditional versus absolute moved space]
		Usually the \defn{moved space} of $x$ is defined as $\mathsf{Mov}(x):=\mathrm{Im}(x- \mathrm{Id}_V)$. Since, taking a $T$-reduced expression $t_1 t_2 \cdots t_k$ of $x$, we have $\mathsf{Mov}(x) \subseteq \sum_{i=1}^k \mathbb{R} \alpha_{t_i}$, we always have $\mathsf{Mov}(x) \subseteq \M(x)$, but the inclusion is strict in general. Equality holds for instance when $W$ is finite: as we shall see in the next proposition, we always have $\dim(\M(x)) = \ell_T(x)$, and by Theorem~\ref{thm_carter}~(1) the same holds for $\mathsf{Mov}(x)$ when $W$ is finite. For a case where the inclusion is strict, let $W$ be of type $\widetilde{A_1}$ with generators $s$ and $t$ and let $x=c=st$. Then $\ell_T(x)= 2$, hence $\M(x)=V$, but $\mathsf{Mov}(x)$ has dimension $1$ since it fixes $\delta = \alpha_s + \alpha_t$ (see Example~\ref{ex_a1_tilde}). 
		
	\end{rmq}
	
	\begin{prop}\label{lem_wm_basic}
		Let $x\in [1,c]_T$. Then $\dim(\M(x)) = \ell_T(x)$. Moreover, for every $T$-reduced expression $t_1 t_2 \cdots t_k$ of $x$, we have $$\M(x) = \mathbb{R} \alpha_{t_1} \oplus \mathbb{R} \alpha_{t_2} \oplus \cdots \oplus \mathbb{R} \alpha_{t_k}.$$ 
	\end{prop}
	
	\begin{proof}
		Given a $T$-reduced expression $(t_1, t_2, \dots, t_k)$ of $x$, by transitivity of the Hurwitz action on $T$-reduced expressions of $c$ (Theorem~\ref{thm_hurw_trans}), we get that $\alpha_{t_1}, \dots, \alpha_{t_k}$ are linearly independent. We thus have $\dim(\M(x)) \geq \ell_T(x)=k$. We show that we have equality, which also shows the second statement. Let $y = x^{-1}c$ and let $t_{k+1} \cdots t_n$ be a $T$-reduced expression of $y$. By Lemma~\ref{lem_form_hurwitz}, for all $i \geq k+1$ and $t\leq_T x$ we have $\varphi_c(\alpha_t, \alpha_{t_i})= 0$: indeed since $t\leq _T x$, we can concatenate a $T$-reduced expression of $x$ in which $t$ appears and $t_{k+1} \cdots t_n$ to get a $T$-reduced expression of $c$, and then apply Lemma~\ref{lem_form_hurwitz}. It follows that $$\M(x) \subseteq \bigcap_{i=k+1}^n \ker(\varphi_c(-, \alpha_{t_i})).$$ But the linear forms $\varphi_c(-, \alpha_{t_i})$ are linearly independent: if $\sum_{i=k+1}^n \lambda_i \varphi_c(-, \alpha_{t_i}) = 0$, then evaluating on $\alpha_{t_{k+1}}$ and applying Lemma~\ref{lem_form_hurwitz} yields $\lambda_{k+1}= 0$, and then iterating one gets $\lambda_{k+1}=\lambda_{k+2}=\dots =\lambda_n=0$. We thus have that $$\dim(\M(x)) \leq \dim\left(\bigcap_{i=k+1}^n \ker(\varphi_c(-, \alpha_{t_i}))\right) = k = \ell_T(x),$$  yielding the second inequality. 
	\end{proof}
	
	\begin{rmq}[Hurwitz transitivity on prefixes]\label{hur_trans_pref}
	While the Hurwitz action is transitive on $T$-reduced expressions of Coxeter elements, for elements of $[1,c]_T$ the transitivity of the Hurwitz action is only conjectural in general. There is an even stronger conjecture, stating that $x$ is a Coxeter element in the reflection subgroup $P(x)$ (see Conjecture~\ref{conj_cox_el}), which would imply Hurwitz transitivity. Apart from finite Coxeter groups~\cite[Section 1.4]{Dual}, this is known to hold more generally in crystallographic Coxeter groups~\cite[Corollary 5.8]{HK} (see also~\cite[Theorem 3.22]{PS} for the affine types), with a proof using representation theory of hereditary algebras, and also in rank three Coxeter systems~\cite[Theorem~3.10]{DPS}. In those cases where the Hurwitz action is transitive on $T$-reduced expressions of $x$, one can easily derive a proof of Lemma~\ref{lem_wm_basic} without using $\varphi_c$ (but proving Hurwitz transitivity is difficult in general, and for proving Theorem~\ref{thm_mapinj}, we need Proposition~\ref{prop:ortho_y} below for which the form $\varphi_c$ is needed).     	
	\end{rmq}
	
	\begin{rmq}[More general absolute moved spaces]\label{rmq:genmov}
		As already pointed out, one may define $\M(x)$ for an arbitrary $x\in W$ exactly as we did for $x\in [1,c]_T$ but it will not be weel-behaved in general and the conclusions of Lemma~\ref{lem_wm_basic} will fail--see Example~\ref{ex_bad_behav_am} below. One can show (using a suitable Coxeter system in which the ambiant one is a standard parabolic subgroup) that Lemma~\ref{lem_wm_basic} can be extended to those elements $x \in W$ admitting a $T$-reduced expression whose factors yield a linearly independent set of roots\footnote{I thank Jean-Yves Hée for pointing this out to me}, but the construction is more involved and we shall not use it. In the following lemma we give an elementary argument in the case where the Tits form $B(-, -)$ is nondegenerate.  
	\end{rmq}
	
	\begin{lemma}
	Let $(W,S)$ be a Coxeter system and assume that the Tits form $B(-,-)$ is nondegenerate. Let $x = t_1 t_2 \cdots t_k$ be a $T$-reduced expression of an element $x\in W$ and assume that the roots $\{\alpha_{t_i}\}_{i=1}^k$ are linearly independent. Then $$\mathsf{Mov}(x) = \mathbb{R} \alpha_{t_1} \oplus \mathbb{R} \alpha_{t_2} \oplus \cdots \oplus \mathbb{R} \alpha_{t_k}.$$
	\end{lemma}
	
	\begin{proof}
	We have $\mathsf{Mov}(x) \subseteq \mathbb{R} \alpha_{t_1} \oplus \mathbb{R} \alpha_{t_2} \oplus \cdots \oplus \mathbb{R} \alpha_{t_k}$. But $\dim(\mathsf{Mov}(x)) = \dim(V) - \dim (\ker(x - \mathrm{Id}_V))$. Let $v\in V$ such that $x(v) = v$. Then $t_1(v) = t_2 t_3 \cdots t_k(v)$ and since the $\alpha_{t_i}$'s are linearly independent, this forces $t_1(v) = v$. Repeating we get that $t_1(v) = t_2(v) = \cdots = t_k(v)$, hence that $\ker(x - \mathrm{Id}_V) = H_{t_1} \cap H_{t_2} \cap \cdots \cap H_{t_k}$. Since $B(-,-)$ is nondegenerate and the $\alpha_{t_i}$'s are linearly independent, we have $\dim( \bigcap_{i=1}^k H_{t_i}) = n-k$, hence $\dim(\mathsf{Mov}(x)) = k$, which concludes the proof.  
	\end{proof}
	
	\begin{exple}\label{ex_bad_behav_am}
		We give examples of Coxeter groups and elements $x$ for which part of the conclusion of Proposition~\ref{lem_wm_basic} fails. In these examples we define $\M(x)$ as in Definition~\ref{def:wms} even if $x$ is not in a noncrossing partition poset. 
		\begin{enumerate}
			\item It has been shown by Duszenko~\cite{Dus} that for irreducible, non-finite and non-affine Coxeter systems, the reflection length is unbounded. Taking such a Coxeter system $(W, S)$ and $x\in W$ such that $\ell_T(x) > |S|=\dim(V)$ (which can occur even in affine type), we cannot have $\dim(\M(x)) = \ell_T(x)$. 
			\item Let $W$ be of type $\widetilde{A_3}$, with generators $s, t, u, v$ and $m_{st} = m_{tu} = m_{uv} = m_{vs}=3$, $m_{su}=m_{tv}=2$. Let $x= (sts)(tut)(uvu)(vsv)$. Using Theorem~\ref{thm_dyer_ref_length} one checks that $\ell_T(x)=4$. But $\{ \alpha_{sts}, \alpha_{tut}, \alpha_{uvu}, \alpha_{vsv}\}$ span the three-dimensional subspace $U=\{ \sum_{r\in S} \lambda_r \alpha_r \ \vert \ \lambda_s + \lambda_u = \lambda_t + \lambda_v\}$. One also checks that $x = t (stuts) (uvsvu) v$, and the roots $\{\alpha_{t}, \alpha_{stuts}, \alpha_{uvsvu}, \alpha_{v}\}$ span the three-dimensional subspace $U' = \{ \sum_{r\in S} \lambda_r \alpha_r \ \vert \ \lambda_s = \lambda_u\}$, distinct from $U$. Since $U + U' = V$ we have $\M(x) = V$ which has dimension four, but none of the above two $T$-reduced expressions have their corresponding roots spanning $V$.  
		\end{enumerate}
	\end{exple}
	
	\begin{prop}\label{prop:ortho_y}
		Let $x\in [1, c]_T$. Let $y = x^{-1} c$. Let $t_1t_2\cdots t_k$ be a $T$-reduced expression of $x$. Then $$\M(y) = \bigcap_{i=1}^k  \ker(\varphi_c(\alpha_{t_i}, -))=\{v\in V \ \vert \ \varphi_c(u, v) = 0 \ \forall u \in \M(x)\}.$$
	\end{prop}
	
	\begin{proof}
		Let $t_{k+1}\cdots t_n$ be a $T$-reduced expression of $y$. Then by the previous lemma $(\alpha_{t_{k+1}}, \dots, \alpha_{t_n})$ forms a basis of $\M(y)$. Since $t_1 t_2 \cdots t_n$ is a $T$-reduced expression of $c$, for all $j \geq k+1$ and all $i \leq k$ by Lemma~\ref{lem_form_hurwitz} we have $\varphi_c(\alpha_{t_i}, \alpha_{t_j})=0$. This shows that $\M(y) \subseteq \bigcap_{i=1}^k  \ker(\varphi_c(\alpha_{t_i}, -))$. We get the reversed inclusion by comparing dimensions: we have $\dim(\M(y))= \ell_T(y) = n-k$, while the linear forms $\varphi_c(\alpha_{t_i}, -)$ are linearly independent (one argues as in the proof of Proposition~\ref{lem_wm_basic}). This yields that $\dim(\bigcap_{i=1}^k  \ker(\varphi_c(\alpha_{t_i}, -))) = n-k$. We thus have the first equality of sets. 
		
		Let $v \in \M(y)$. Then keeping the above notation, we have that $v$ is a linear combination of $\alpha_{t_{k+1}}, \dots, \alpha_{t_n}$, while if $u\in \M(x)$ then $u$ is a linear combination of $\alpha_{t_1}, \dots, \alpha_{t_k}$. By Lemma~\ref{lem_form_hurwitz} this yields that $\varphi_c(u, v) = 0$. Conversely, let $v\in V$ such that $\varphi_c(u,v) = 0$ for all $u\in \M(x)$, and let $\lambda_i\in \mathbb{R}$ such that $v = \sum_{i=1}^n \lambda_i \alpha_{t_i}$. We have $\alpha_{t_1} \in \M(x)$, hence $$ 0 = \varphi_c(\alpha_{t_1}, v) = \sum_{i=1}^n \lambda_i \varphi_c( \alpha_{t_1}, \alpha_{t_i}) = \lambda_1,$$
		and iterating with $\alpha_{t_2}, \dots, \alpha_{t_k}$ yields $\lambda_1 = \lambda_2 = \dots = \lambda_k = 0$. Hence $v$ is a linear combination of $\alpha_{t_{k+1}}, \dots, \alpha_{t_n}$, and thus $v\in \M(y)$.   
	\end{proof}
	
	\begin{cor}\label{cor_det}
		Let $x\in [1, c]_T$. Let $y = x^{-1} c$. Then $\M(y)$ is determined by $\M(x)$. 
	\end{cor}
	
	\begin{proof}
		This is immediate with the equality $$\M(y) =\{v\in V \ \vert \ \varphi_c(u, v) = 0 \ \forall u \in \M(x)\}.$$
	\end{proof}
	
	\section{Various injective maps from noncrossing partition posets}\label{sec:inj}
	
	\subsection{Subspaces of the geometric representation}\label{sec:subspaces}

	The aim of this subsection is to show the first point of Theorem~\ref{thm_mapinj}: 
	
	\begin{theorem}\label{thm_inj_subspaces}
		Let $(W,S)$ be a Coxeter system and $c\in W$ a Coxeter element. The map $$[1,c]_T \longrightarrow \mathrm{Sub}(V), \ x \mapsto \M(x),$$ is order-preserving and injective. 
	\end{theorem}
	
	\begin{proof}
		Let $x, y\in [1,c]_T$ such that $x \leq_T y$. Then there is a $T$-reduced expression $t_1 t_2 \cdots t_k$ of $y$ such that $t_1 \cdots t_i$ is a $T$-reduced expression of $x$ for some $i=0, \dots, k$. We then have $$ \M(x) = \bigoplus_{j=1}^i \mathbb{R} \alpha_{t_j} \subseteq \bigoplus_{j=1}^k \mathbb{R} \alpha_{t_j} = \M(y).$$ The map is thus order-preserving. 
		\medskip
		
		We now show that it is injective. Assume that $x, x'\in  [1, c]_T$ are such that $\M(x) = \M(x')$. Let $y=x^{-1}c$, $y'=x'^{-1}c$. By Corollary~\ref{cor_det}, we have $\M(y) = \M(y')$. Let $t_1 t_2 \cdots t_n$ be a $T$-reduced expression of $c$ such that $t_1 t_2 \cdots t_i$ is a $T$-reduced expression of $x$ for some $i=0, \dots, n$ (and hence $t_{i+1} \cdots t_n$ is a $T$-reduced expression of $y$). We then have $\M(x)\oplus \M(y)=V$ and $$\M(x) = \M(x') = \bigoplus_{j=1}^i \mathbb{R} \alpha_{t_j}, \ \M(y) = \M(y') = \bigoplus_{j=i+1}^n \mathbb{R} \alpha_{t_j}.$$
		We have $xy = c = x'y'$, hence $x^{-1} x' = y y'^{-1}$. Now all $T$-reduced expressions of $y$ and $y'$ have their factors having their corresponding roots in $\M(y)=\M(y')$, and $T$-reduced expressions of $x$ and $x'$ have their factors having their corresponding roots in $\M(x)=\M(x')$. This means that $x^{-1} x'$ (respectively $y y'^{-1}$) can be written as a (not necessarily reduced) product of reflections whose roots all lie in $\M(x)$ (resp. in $\M(y)$). We thus have $\mathrm{Im}(x^{-1}x' - \mathrm{Id}_V) \subseteq \M(x)$ and $\mathrm{Im}(y y'^{-1} - \mathrm{Id}_V) \subseteq \M(y)$. Since $x^{-1}x'= y y'^{-1}$ and $\M(x) \cap \M(y) = \{0_V\}$, this forces $\mathrm{Im}(x^{-1}x' - \mathrm{Id}_V) = \{0_V\}$, hence $x^{-1} x'=1$. We thus get $x=x'$. 
	\end{proof}
	
	\begin{exple}\label{ex:fail_poset}
		When $W$ is finite, we have $\mathsf{Mov}(x)=\M(x)$ and the map $x \mapsto \M(x)$ is an isomorphism of posets onto its image, as a corollary of the main result of~\cite{BW_ortho} (see Proposition~\ref{prop_bw}~(2)). Unfortunately, this does not hold in general. As a counterexample (inspired from~\cite[Remark 2.13]{RS}), consider $W$ of type $\widetilde{A_3}$, with generators $s_1, s_2, s_3, s_4$ and $m_{12} = m_{23} = m_{34} = m_{41}=3$, $m_{13}=m_{24}=2$. Let $c= s_1 s_2 s_3 s_4$. We consider the following two $T$-reduced expressions of $c$: \begin{align*}s_1 s_2 s_3 s_4 & = (s_1 s_2 s_1) (s_1 s_4 s_1) s_1 (s_4 s_3 s_4)\\
			&= (s_1 s_2 s_1) (s_1 s_4 s_1) (s_4 s_3 s_4) (s_3 s_4 s_1 s_4 s_3),\\
			s_1 s_2 s_3 s_4 &= s_1 (s_2 s_3 s_2) s_2 s_4 = (s_2 s_3 s_2) (s_3 s_2 s_1 s_2 s_3) s_2 s_4. 
		\end{align*} 
		We thus have that $x= (s_1 s_2 s_1) (s_1 s_4 s_1) (s_4 s_3 s_4) \leq_T c$ and $t= s_2 s_3 s_2 \leq_T c$. We have $\ell_T(x) = 3$ and $$\M(x) = \mathbb{R} (\alpha_1 + \alpha_2) \oplus \mathbb{R} (\alpha_1 + \alpha_4 ) \oplus \mathbb{R} (\alpha_3 + \alpha_4) = \left\{ \sum_{i=1}^4 a_i \alpha_i \in V \mid a_1 + a_3 = a_2 + a_4\right\},$$
		and $$\M(t) = \mathbb{R} (\alpha_2 + \alpha_3) \subseteq \M(x),$$
		but one checks (for instance using Theorem~\ref{thm_dyer_ref_length}) that $t \not\leq_T x$. 
	\end{exple}
	
	While Example~\ref{ex:fail_poset} gives an example of a Coxeter group of rank four where the map $x \mapsto \M(x)$ is not an isomorphism of posets onto its image, the property trivially holds in rank $2$, and the following result shows that it also holds in rank $3$: 
	
	\begin{prop}\label{prop_poset_rk3}
	Let $(W,S)$ be a Coxeter system of rank $3$, and $c\in W$ a Coxeter element. The map $$[1,c]_T \longrightarrow \mathrm{Sub}(V), \ x \mapsto \M(x),$$ is an isomorphism of posets onto its image.   
\end{prop}

\begin{proof}
By Theorem~\ref{thm_inj_subspaces}, we already know that the map is order-preserving and injective. Hence let $u, v\in [1,c]_T$ such that $\M(u) \subseteq \M(v)$. We must show that $u \leq_T v$. If $u=1$ or $v=c$, this trivially holds. If $\M(u) = \M(v)$, then by Theorem~\ref{thm_inj_subspaces} we have $u=v$. We can thus assume that $u \neq v$, and that $\dim(\M(u))=\ell_T(u) = 1$, $\dim(\M(v))=\ell_T(v)=2$. It follows that $\M(u)$ is a line included in the plane $\M(v)$. Since $\ell_T(u) = 1$, we have $u\in T$, hence $\M(u) = \mathbb{R} \alpha_u$. Letting $t_1 t_2$ be a $T$-reduced expression of $v$, we have $\alpha_u \in \M(v) = \mathbb{R} \alpha_{t_1} \oplus \mathbb{R} \alpha_{t_2}$. By~\cite[Remark 3.2]{Dyer_Bruhat}, $\Phi \cap \M(v)$ is the set of roots of a dihedral reflection subgroup of $W$. We thus have that $u, t_1, t_2$ are reflections in a common dihedral reflection subgroup of $W$, with $t_1 \neq t_2$. This forces $ u t_1 t_2$ to be a reflection $t'$ of this subgroup, hence of $W$. We thus have $v = t_1 t_2 = u t'$, showing that $u \leq_T v$. 
\end{proof}
	
	\subsection{Reflection subgroups}
	
	Let $(W,S)$ be a Coxeter system and $x\in W$. Recall the notation $$P(x):= \langle t\in T \mid t\leq_T x \rangle.$$
	When $W$ is finite, this subgroup is parabolic, whence the notation. In general $P(x)$ is not a parabolic subgroup of $W$, but only a reflection subgroup (see Example~\ref{ex_nonparab}). 
	
	\medskip
	
	Before proving our main result we collect some properties of the subgroups $P(x)$ that we shall use later on. 
	
	\begin{lemma}\label{lemm_p_prop}
	Let $(W,S)$ be a Coxeter system and $c\in W$ a Coxeter element. 
	\begin{enumerate}
		\item We have $P(c) = W$; more generally, if $c'$ is a standard parabolic Coxeter element such that $c' \leq_T c$, then $P(c')$ is the standard parabolic subgroup $\langle \mathrm{supp}(c') \rangle$, where $\mathrm{supp}(c')$ denotes the set of those elements in $S$ occurring in any $S$-reduced expression of $c'$,
		\item Let $x \in [1,c]_T$ and $w\in W$ such that $w x w^{-1} \leq_T c$. Then $P(w x w^{-1}) = w P(x) w^{-1}$,
		\item If $P(x)$ is finite, then $P(x)$ is the parabolic closure of $x$,
		\item If $t, t'\in T$, $t\neq t'$, such that $tt' \leq_T c$, then $P(tt')$ is the unique maximal dihedral reflection subgroup containing $t$ and $t'$. In this case we have $P(tt') \cap T = \{q \in T \mid q \leq_T tt'\}$.  
	\end{enumerate}
	\end{lemma}
	
	\begin{proof}
	\begin{enumerate}
		\item For $c$, this is immediate, since $s\leq_T c$ for all $s\in S$ (alternatively $P(c)$ is a reflection subgroup of $W$ containing $c$, hence by Theorem~\ref{thm_cox_proper} it can only be $W$). Now let $c'$ be a standard parabolic Coxeter element. Let $s'\in \mathrm{supp}(c')$. Then $s' \leq_T c'$, hence $\langle\mathrm{supp}(c') \rangle \subseteq P(c')$. To conclude the proof we show that for every $t\in T$ such that $t \leq_T c'$, we have $t\in \langle \mathrm{supp}(c')\rangle$. In this case by Theorem~\ref{thm_inj_subspaces} we have $\alpha_t \in \M(t) \subseteq \M(c') = V'$, where $V'$ denotes the geometric representation of the standard parabolic subgroup $\langle \mathrm{supp}(c') \rangle$. This forces $t\in \langle \mathrm{supp}(c') \rangle$, since $V' \cap \Phi$ precisely consists of the roots whose reflections are in $\langle \mathrm{supp}(c') \rangle$ (see for instance~\cite[Proposition 3.3]{BH}).
		\item Since $T$ is stable by conjugation, if $t\in T$ satisfies $t \leq_T x$, then $w t w^{-1} \leq_T w x w^{-1}$. Hence $w P(x) w^{-1} \subseteq P(w x w^{-1})$. Swapping the roles of $x$ and $w x w^{-1}$ yields $w^{-1} P (w x w^{-1}) w \subseteq P(x)$, yielding the reversed inclusion.
		\item Assume that $P(x)$ is finite. Then there is $w\in W$ such that $w P(x) w^{-1}$ is included in a finite standard parabolic subgroup $W'$ of $W$ (see for instance~\cite[Proposition 1.3]{BH_a}). It suffices to show that $w P(x) w^{-1}$ is the parabolic closure of $w x w^{-1}$ inside $W'$. Indeed, it will therefore also be parabolic inside $W$, hence $P(x)$ as well since parabolic subgroups are stable by conjugation; moreover, if $x \in P \subsetneq P(x)$ with $P$ parabolic, then $w P w^{-1}$ would be a proper parabolic subgroup of $w P(x) w^{-1}$ containing $wxw^{-1}$, a contradiction. 
		
		Now by properties recalled in Subsection~\ref{various}, the parabolic closure of $w x w^{-1}$ in the finite Coxeter group $W'$ is the reflection subgroup having as reflections all the reflections $t\in T'$ such that $t\leq_{T'} w x w^{-1}$. But in fact, if $t \in T$ is such that $t \leq_T wxw^{-1}$, then $t\in T'$: indeed $T$-reduced decompositions of elements lying in parabolic subgroups have all their factors in this parabolic subgroup (see~\cite[Theorem 1.4]{BDSW}). We thus have 
		\begin{align*}
			\{t\in T' \mid t\leq_{T'} wxw^{-1}\} &= \{t\in T \mid t\leq_{T} wxw^{-1}\} \\ &=
			\{ t\in T \mid w^{-1} t w \leq_T x\} \\ &= w\{ t\in T \mid t \leq_T x \} w^{-1}, 
		\end{align*}
		and the first set generates the parabolic closure of $wxw^{-1}$ inside $W'$, while the last set generates $w P(x) w^{-1}$. The two groups thus coincide.   
		\item Denote by $W(t,t')$ the unique maximal dihedral reflection subgroup of $W$ containing $t$ and $t'$. By~\cite[Theorem 1.6 and Proposition 1.4]{Gobet_max_dih}, the set of reflections of $W(t, t')$ is precisely those $q\in T$ such that $q tt'$ is a reflection of $W$, that is, such that $q\leq_T tt'$. This shows all claims. 
	\end{enumerate}
	\end{proof}
	
	We now prove the main result of the subsection. Denote by $\mathrm{RS}(W)$ the set of reflection subgroups of $W$, partially ordered by inclusion. 
	
	\begin{theorem}\label{poset_2}
		Let $(W,S)$ be a Coxeter system and $c\in W$ a Coxeter element. The map $$[1,c]_T \longrightarrow \mathrm{RS}(W), \ x \mapsto P(x),$$ is order-preserving and injective. 
	\end{theorem}
	
	\begin{proof}
		The fact that the map is order-preserving follows from the fact that if $u, v\in [1, c]_T$ with $u \leq_T v$, then for every $t\in T$ such that $t\leq_T u$, we have $ t\leq_T v$. 
		
		Let $x, x'\in [1,c]_T$ such that $P(x) = P(x')$. The set of reflections of $P(x)$ is given by $P(x) \cap T$, and since $P(x)$ is by definition generated by the set $T(x)$ of reflections lying below $x$ for the absolute order, by~\cite[Corollary 3.11]{Dyer_subgroups} we have $P(x) \cap T = \bigcup_{w\in P(x)} w T(x) w^{-1}$. Hence every reflection in $P(x)$ can be written as a palindromic word in elements of $T(x)$, and in particular, every reflection in $P(x)$ has its root which is a linear combination of roots whose corresponding reflections are in $T(x)$. We thus have   
		$$\M(x) = \sum_{t\in T(x)} \mathbb{R} \alpha_t = \sum_{t\in P(x)\cap T} \mathbb{R} \alpha_t = \sum_{t\in P(x')\cap T} \mathbb{R} \alpha_t = \sum_{t\in T(x')} \mathbb{R} \alpha_t = \M(x').$$
		It then follows from Theorem~\ref{thm_inj_subspaces} that $x=x'$, which concludes the proof. 
	\end{proof}
	
	One can wonder whether the map $x \mapsto P(x)$ is an isomorphism of posets onto its image or not. Unlike for $x \mapsto \M(x)$ we do not have a counterexample to this statement. In Example~\ref{ex:fail_poset}, one easily checks using for instance the realization of $\widetilde{A_3}$ as a group of affine permutations that $t \notin P(x)$. In fact $P(x) = \langle  s_1 s_2 s_1, s_4 s_3 s_4 \rangle\times \langle s_1 s_4 s_1\rangle \cong \widetilde{A}_2 \times A_1$, while $\langle P(x), t \rangle = \langle  s_1 s_2 s_1, s_4 s_3 s_4 \rangle \times \langle s_1 s_4 s_1, s_2 s_3 s_2\rangle \cong \widetilde{A}_2 \times \widetilde{A}_2$.   
	
	We conjecture that the map $x \mapsto P(x)$ is an isomorphism of posets onto its image (see Conjecture~\ref{conj_isom_3maps} below).

	\subsection{Sets of reflections}
	Let $(W,S)$ be a Coxeter system and $c\in W$ a Coxeter element. Let $x\in [1,c]_T$. As in the previous subsection let $$T(x):= \{t\in T \mid t \leq_T x\}.$$
	When $W$ is finite, one has $P(x) \cap T = T(x)$ (see Subsection~\ref{ex_a1_tilde}), but in general $T(x) \subsetneq P(x) \cap T$. 
		\begin{theorem}\label{poset_3}
		Let $(W,S)$ be a Coxeter system and $c\in W$ a Coxeter element. The map $$[1,c]_T \longrightarrow \mathcal{P}(T), \ x \mapsto T(x),$$ is order-preserving and injective. 
	\end{theorem}
	
	\begin{proof}
	Let $x, y\in [1,c]_T$ such that $x \leq_T y$. If $t\in T(x)$, then there is a $T$-reduced expression of $x$ starting by $t$. Extending it into a $T$-reduced expression of $y$ yields that $t \in T(y)$. We thus have $T(x) \subseteq T(y)$, and the map is order-preserving. 
	
	Let $x, x'\in [1,c]_T$ such that $T(x) = T(x')$. Then $$P(x) = \langle T(x) \rangle = \langle T(x') \rangle = P(x').$$ By Theorem~\ref{poset_2} we have $x=x'$, and the map is injective. 
	\end{proof}
	
		Here as well, one can wonder whether the map $x \mapsto P(x)$ is an isomorphism of posets onto its image. Unlike for $x \mapsto \M(x)$ we do not have a counterexample to this statement. In Example~\ref{ex:fail_poset}, we have $t \not\leq x$, hence $t\notin T(x)$.   
	
	We conjecture that the map $x \mapsto T(x)$ is an isomorphism of posets onto its image (see Conjecture~\ref{conj_isom_3maps} below).
	
\subsection{Link between the various maps}

Let $x\in [1,c]_T$. We have $P(x) = \langle T(x) \rangle$, and as observed in the proof of Theorem~\ref{poset_2} we have $\M(x) = \sum_{t\in P(x)\cap T} \mathbb{R} \alpha_t$. For $x, y\in [1,c]_T$, we thus have the following implications: \begin{equation}\label{impl} T(x) \subseteq T(y) \Rightarrow  P(x) \subseteq P(y) \Rightarrow \M(x) \subseteq \M(y).\end{equation}
We deduce the following 

\begin{lemma}\label{prop_link_maps}
Let $(W,S)$ be a Coxeter system and $c\in W$ a Coxeter element. 
\begin{enumerate}
\item If the map $x \mapsto P(x)$ is an isomorphism of posets onto its image, then the map $x \mapsto T(x)$ is an isomorphism of posets onto its image. 
\item If the map $x \mapsto \M(x)$ is an isomorphism of posets onto its image, then both maps $x \mapsto P(x)$ and $x \mapsto T(x)$ are isomorphisms of posets onto their images.
\end{enumerate}
\end{lemma}

\begin{proof}
We already known from Theorems~\ref{poset_2} and~\ref{poset_3} that both maps $x \mapsto P(x)$ and $x \mapsto T(x)$ are injective and order-preserving. For point $1$, it suffices to show that, if $x, y\in [1,c]_T$ with $T(x) \subseteq T(y)$, then $x \leq_T y$. By~\eqref{impl} this implies that $P(x) \subseteq P(y)$, and by assumption this implies $x \leq_T y$. For point $2$ it suffices to show the result for the map $x \mapsto P(x)$ and use point $1$ to deduce it for the map $x \mapsto T(x)$. But if $P(x) \subseteq P(y)$, then by~\eqref{impl} we have $\M(x) \subseteq \M(y)$, hence $x\leq_T y$ by assumption. 
\end{proof}

\begin{cor}\label{isom_rank_three}
Let $(W,S)$ be a Coxeter system of rank three and $c\in W$ a Coxeter element. Then the three maps $x \mapsto \M(x)$, $x \mapsto P(x)$ and $x \mapsto T(x)$ are isomorphisms of posets onto their images. 
\end{cor}

\begin{proof}
This is a combination of Lemma~\ref{prop_link_maps} and Proposition~\ref{prop_poset_rk3}. 
\end{proof}

\subsection{Link with the lattice property}

While the map $x \mapsto \M(x)$ is not an isomorphism of posets onto its image in general (see Example~\ref{ex:fail_poset}), we expect the maps $x \mapsto P(x)$ and $x \mapsto T(x)$ to be isomorphisms of posets onto their images. 

We show that, if $[1,c]_T$ is a lattice, then this holds for the map $x \mapsto T(x)$: 

\begin{prop}
Let $(W,S)$ be a Coxeter system and $c\in W$ a Coxeter element. Assume that $[1,c]_T$ is a lattice. Then the map $x \mapsto T(x)$ is an isomorphism of posets onto its image. 
\end{prop}

\begin{proof}
Since, by Theorem~\ref{poset_3}, the map $x \mapsto T(x)$ is injective and order-preserving, it suffices to show that, if $x, y\in [1,c]_T$ satisfy $T(x) \subseteq T(y)$, then $x \leq_T y$. First note that, since $T(x) \subseteq T(y)$, by~\eqref{impl} we have $\M(x) \subseteq \M(y)$, hence $\ell_T(x)=\dim(\M(x)) \leq \dim(\M(y))=\ell_T(y)$. 

We argue by induction on $N(x,y):=|S| - (\ell_T(y) - \ell_T(x))$. If $N(x,y) = 0$, then $y = c$, $x=1$, hence $x \leq_T y$. We can thus assume that $N(x,y) \geq 1$, and that the result holds for $N(x',y') < N(x,y)$. If $x=1$, then $x \leq_T y$. We can thus assume that $x\neq 1$. Let $t\in T$ such that $t \leq_T x$. Since $t \in T(x) \subseteq T(y)$, we have $t\leq_T y$. Let $x'\in [1,c]_T$ such that $tx' = x$. Since every $T$-reduced expression of $x'$ becomes a $T$-reduced expression of $x$ when concatenating $t$ at the left, we have $T(x') \subseteq T(x) \subseteq T(y)$, and $\ell_T(x') = \ell_T(x)-1$. Hence $N(x',y) = N(x,y)-1$. By induction, we have $x'\leq_T y$. We thus have $$ t\leq_T x, \ x' \leq_T x, \ t \leq_T y, \ x' \leq_T y.$$  
Since $x= tx'$ and $\ell_T(x) = \ell_T(x') + 1$, we have that $x$ is the join of $x'$ and $t$. Since $x' \leq_T y$ and $t \leq_T y$, we deduce that $x \leq_T y$.  
\end{proof}

	\section{Applications}\label{sec:applications}
	
	In this section, we use absolute moved spaces to show properties of noncrossing partition posets $[1,c]_T$. 
	
	\subsection{Some properties of noncrossing partition posets holding in general}
	
	We derive from the results of the previous sections some properties of noncrossing partition posets holding in full generality, that is, for an arbitrary Coxeter system $(W,S)$ and an arbitrary choice of Coxeter element $c\in W$. 
	
	\begin{lemma}\label{lem_2_ref_comm}
	Let $(W,S)$ be a Coxeter system and $c\in W$ a Coxeter element. Let $t, t' \in T$ such that $t\neq t'$, $tt' \leq_T c$, and $t't \leq_T c$. Then $tt'=t't$. 
	\end{lemma}
	\begin{proof}
	We have $\M(tt') = \mathbb{R} \alpha_t \oplus \mathbb{R} \alpha_{t'} = \M(t't)$, which by Theorem~\ref{thm_inj_subspaces} forces $t't=tt'$. 
	\end{proof}
	
	\begin{prop}\label{prop_sets}
	Let $(W,S)$ be a Coxeter system and $c\in W$ a Coxeter element. Let $x\in [1,c]_T$. Let $t_1 t_2 \cdots t_k$ be a $T$-reduced expxression of $x$. Assume that there is $\sigma \in \mathfrak{S}_k$ such that $t_{\sigma(1)} t_{\sigma(2)} \cdots t_{\sigma(k)}$ is a $T$-reduced expression of an element $y\in [1,c]_T$. Then $x=y$, and the two expressions $t_1 t_2 \cdots t_k$ and $t_{\sigma(1)} t_{\sigma(2)} \cdots t_{\sigma(k)}$ are related by a sequence of commutation of adjacent letters. In particular, given a set $\{t_1, t_2, \dots, t_k\}$ of $k$ distinct reflections ($k\geq 1$), there is at most one element from $[1,c]_T$ having a $T$-reduced expression whose letters are exactly the $t_i$'s. 
	\end{prop}
	
	\begin{proof}
	We have $$\M(x) = \bigoplus_{i=1}^k \mathbb{R} \alpha_{t_i} = \M(y),$$ hence Theorem~\ref{thm_inj_subspaces} yields $x=y$. 
	
	For the second statement, it is enough to show that, if $i < j$ and $t_i$ appears after $t_j$ in $t_{\sigma(1)} t_{\sigma(2)} \cdots t_{\sigma(k)}$, then $t_i t_j = t_j t_i$. But this means that $t_i t_j \leq_T x$ and $t_j t_i \leq_T y = x$, hence Lemma~\ref{lem_2_ref_comm} concludes the proof. 
	\end{proof}

Proving that $[1,c]_T$ is not a lattice amounts to exhibiting a so-called \textit{bowtie} in the poset, that is, a quadruple $\{a,b,x,y\}$ of distinct elements such that $a$ and $b$ are maximal lower bounds for $x$ and $y$ and $x$ and $y$ are minimal upper bounds for $a$ and $b$ (see~\cite[Section 10]{McCammond}). In all known examples, one observes that it never happens that such bowties have their elements occurring in successive ranks, that is, one never has $\ell_T(a) = k = \ell_T(b)$ and $\ell_T(x) = k+1 = \ell_T(y)$ for some $1 \leq k \leq n-2$. The next proposition shows that this is a general property, holding in an arbitrary Coxeter system:

	\begin{prop}[Absence of bowties in successive ranks]\label{prop:bow_succ}
	Let $(W,S)$ be a Coxeter system and $c\in W$ a Coxeter element. There exists no quadruple $\{a, b, x, y\}$ of pairwise distinct elements of $[1,c]_T$ such that 
	\begin{itemize}
		\item $\ell_T(a) = \ell_T(b) = k$, $\ell_T(x) = \ell_T(y) = k+1$, for some $1 \leq k \leq n-2$, 
		\item $a \leq_T x, a \leq_T y, b \leq_T x, b \leq_T y$. 
	\end{itemize}
	\end{prop}
	
	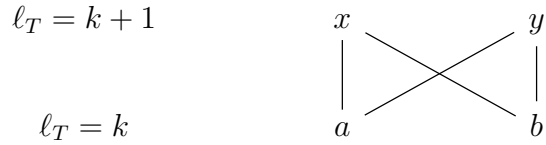
\begin{figure}[h!]
			\[\begin{tikzcd}
			{\ell_T = k+1} && x && y \\
			{\ell_T = k} && a && b
			\arrow[no head, from=2-3, to=1-3]
			\arrow[no head, from=2-3, to=1-5]
			\arrow[no head, from=2-5, to=1-3]
			\arrow[no head, from=2-5, to=1-5]
		\end{tikzcd}\]
		\caption{A bowtie in successive ranks as in Proposition~\ref{prop:bow_succ}}.
		\label{bowtie_succ}
	\end{figure}
	
	\begin{proof}
	Assume for contradiction that such a configuration of four elements exists. We consider the four subspaces $\M(a), \M(b), \M(x), \M(y)$ of $V$. By Theorem~\ref{thm_inj_subspaces}, we have $\dim(\M(a)) = \dim(\M(b)) = k$, $\dim(\M(x)) = \dim(\M(y)) = k+1$, and these subspaces are pairwise distincts. Since the map $z \mapsto \M(z)$ is order-preserving, we also have $\M(a) \subseteq \M(x)$, and $\M(a) \subseteq \M(y)$, and the same holds with $a$ replaced by $b$. Since $\M(x) \neq \M(y)$ and they have the same dimension, namely $k+1$, this forces $\M(x) \cap \M(y)$ to have dimension at most $k$, and since $\M(a)$ in included in both and has dimension $k$, we have $\M(x) \cap \M(y) = \M(a)$. The same holds with $a$ replaced by $b$, hence $\M(a) = \M(b)$, contradicting the injectivity of the map $z \mapsto \M(z)$ from Theorem~\ref{thm_inj_subspaces}. 
	\end{proof}
	
	\subsection{Cactus bases}
	
	Cactus bases we introduced by Bellingeri, Godelle and Paris~\cite[Section 3.1]{trickles} as a data from which one can define a generalization of a cactus group. 
	
	\begin{definition}[{\cite[Section 3.1]{trickles}}]\label{def:cactus}
	Let $G$ be a group. A \defn{cactus basis} (based on $G$) is a non-empty family $\mathcal{F}$ of pairs $X=( G_X, \Delta_X)$ were $G_X$ is a subgroup of $G$ and $\Delta_X\in G_X$, such that the following two conditions are satisfied:
	\begin{enumerate}
		\item $G_X \neq G_Y$ if $X \neq Y$, $X, Y\in \mathcal{F}$, 
		\item For $X, Y\in \mathcal{F}$, if $G_Y \subseteq G_X$, then $$(\Delta_X G_Y \Delta_X^{-1}, \Delta_X \Delta_Y \Delta_X^{-1})\in \mathcal{F}\text{~and~}(\Delta_X^{-1} G_Y \Delta_X, \Delta_X^{-1} \Delta_Y \Delta_X)\in \mathcal{F}.$$ 
	\end{enumerate}
	\end{definition}
	
	It has been shown in~\cite[Section 3.1, Example 2]{trickles} that, when $W$ is a finite Coxeter group and $c\in W$ is a Coxeter element, the family $\mathcal{F} = \{(P(x), x) \mid x\in [1,c]_T\}$ is a cactus basis. We extend this result to those cases where the map $x \mapsto P(x)$ is an isomorphism of posets onto its image; by Corollary~\ref{isom_rank_three} this also includes all Coxeter systems of rank three with arbitrary choices of Coxeter elements. 
	
	\begin{lemma}
	Let $(W,S)$ be a Coxeter system and $c\in W$ a Coxeter element. Assume that the map $[1,c]_T \longrightarrow \mathrm{RS}(W)$, $x \mapsto P(x)$, is an isomorphism of posets onto its image. Then $$\mathcal{F}:=\{ (P(x), x) \mid x\in [1,c]_T\}$$ is a cactus basis. 
	\end{lemma}
	
	\begin{proof}
	The first condition of Definition~\ref{def:cactus} is satisfied since by Theorem~\ref{poset_2} the map $x \mapsto P(x)$ is always injective. We now check the second condition, where the assumption is needed. Let $x, y\in [1,c]_T$ such that $P(y) \subseteq P(x)$. Then by assumption, we have $ y \leq_T x$. Since $T$ is stable by conjugation, we have $x y x^{-1} \leq_T x x x^{-1} = x$, hence $x y x^{-1} \in [1,c]_T$. Moreover, it follows from point~(2) of Lemma~\ref{lemm_p_prop} that $x P(y) x^{-1} = P(xyx^{-1})$. This shows that $( x P(y) x^{-1}, x y x^{-1})\in \mathcal{F}$. The proof for $( x^{-1} P(y) x, x^{-1} y x)$ is similar.  
	\end{proof}
	
	\subsection{New proof of the lattice property in rank three}
	
		The aim of this subsection is to give a new proof that $[1,c]_T$ is a lattice when $W$ is a Coxeter group of rank three. It is obtained as a corollary of Proposition~\ref{prop:bow_succ}.

	\begin{cor}[New proof of the lattice property in rank three]\label{cor_new_three}
	Let $(W,S)$ be a Coxeter system of rank three and $c\in W$ a Coxeter element. Then the poset $[1,c]_T$ is a lattice. 
	\end{cor}
	
	\begin{proof}
	Since $\ell_T(1)= 0$ and $\ell_T(c) = 3$, the only obstruction to the lattice property can come from a quadruple $a, b, x, y$ of elements as in Proposition~\ref{prop:bow_succ}, with $k=1$. Proposition~\ref{prop:bow_succ} guarantees that such a quadruple cannot exist.  
	\end{proof}
	
	\begin{rmq}
	The proof of the lattice property of $[1,c]_T$ in rank three that we gave in~\cite{Gobet_max_dih} is obtained from the more general statement that in an abritrary Coxeter group $(W,S)$, given any pair $u, v\in W$ such that $u \leq_T v$ and $\ell_T(v) = \ell_T(u) + 3$, the poset $[u,v]_T$ is a lattice. The above corollary is then obtained by considering $W$ of rank three and setting $u=1, v=c$. The proof of Proposition~\ref{prop:bow_succ} only holds inside a noncrossing partition poset, but the quadruple does not need to lie inside an interval of rank three, hence it gives more information on $[1,c]_T$ than our previous proof.  
	\end{rmq}
	
	\subsection{New families of noncrossing partition posets failing to be lattices}
	
	The aim of this subsection is to show Theorem~\ref{thm_fail_lattice}. To this end, we exhibit bowties in noncrossing partition posets $[1,c]_T$ for infinitely many Coxeter groups of rank $4$. As a consequence, for every Coxeter system $(W,S)$ containing such a rank four subgroup as a standard parabolic subgroup, and compatible choices of Coxeter elements in the sense of the next lemma, the lattice property of $[1,c]_T$ fails (Corollary~\ref{cor_fail}).   
	
	\begin{lemma}\label{lem_parab_fail}
	Let $(W, S)$ be a Coxeter system and let $(W', S')$ be a standard parabolic subgroup of $(W,S)$. Let $c$ be a Coxeter element of $W$ and let $c'$ be the Coxeter element of $W'$ obtained by restricting $c$, that is, by removing from and $S$-reduced expression of $c$ the letters which are not in $S'$. If $[1,c']_{T'}$ fails to be a lattice, then $[1,c]_T$ fails to be a lattice. 
	\end{lemma}
	
	\begin{proof}
	Denote by $T'$ the set $W'\cap T = \bigcup_{w\in W'} w S' w^{-1}$ of reflections of $W'$ and by $V'\subseteq V$ the geometric representation of $(W',S')$. We have $c' \leq_T c$, and $\ell_{T'}(c') = |S'| = \ell_T(c')$. Now let $t\in T$ such that $ t\leq_T c'$. By definition of absolute moved spaces we have $$\alpha_t \in \mathbb{R} \alpha_t = \M(t) \subseteq \M(c') = V'$$ which, since $W'$ is standard parabolic, forces $t$ to lie in $W'$ (see~\cite[Proposition 3.3]{BH}). We deduce that $[1, c']_T = [1, c']_{T'}$. This implies that, if the lattice property fails inside $[1,c']_{T'}$, then it also fails inside $[1,c]_T$. 
	\end{proof}
	
	Consider a Coxeter group $W$ of rank $4$ with a diagram of the form \[\begin{tikzcd}
		s_1 && s_2 \\
		\\
		s_4 && s_3
		\arrow["{2k+1}", no head, from=1-1, to=1-3]
		\arrow["\ell", no head, from=1-3, to=3-3]
		\arrow["\ell", no head, from=3-1, to=1-1]
		\arrow["{2k' +1}", no head, from=3-3, to=3-1]
	\end{tikzcd}\]
	where $\ell \geq 3$, $k,k' \geq 1$. Let $c= s_3 s_1 s_2 s_4$ The main result of the subsection is the following: 
	
	\begin{theorem}\label{thm_fail_lattice}
	Let $W$ be a Coxeter group of rank $4$ with a diagram of the above form. Then $[1,c]_T$ is not a lattice. 
	\end{theorem}
	
	\begin{proof}
	We choose as reflection $t_1$ (respectively $t_2$) the longest element in the standard dihedral parabolic subgroup $W_1=\langle s_1, s_2 \rangle$ (resp. $W_2=\langle s_3, s_4 \rangle$), that is, $t_1 = (s_1 s_2)^k s_1$ and $t_2 = (s_3 s_4)^{k'} s_3$. We let $w_1 = s_3 t_1 s_4$ and $w_2 = s_1 t_2 s_2$. Setting $t_1'= (s_1 s_2)^{k-1} s_1$ and $t_2'=(s_3 s_4)^{k'-1} s_3$, we have $$ c = s_3 s_1 s_2 s_4 = s_3 t_1 t_1' s_4 = s_3 t_1 s_4 (s_4 t_1' s_4)$$ and $$ c = s_1 s_3 s_4 s_2 = s_1 t_2 t_2' s_2 = s_1 t_2 s_2 (s_2 t_2' s_2).$$ We therefore get that $w_1, w_2\in [1,c]_T$ and $\ell_T(w_1) = \ell_T(w_2) = 3$. Moreover we have $w_1 = t_1 (t_1 s_3 t_1) s_4$ and also $w_1 = s_3 s_4 (s_4 t_1 s_4) = t_2 t_2' (s_4 t_1 s_4)$, which yields $t_1, t_2\leq_T w_1$, and similarly we have $t_1, t_2 \leq_T w_2$. We thus have a configuration 			\[\begin{tikzcd}
		{\ell_T = 3} && w_1 && w_2 \\
		{\ell_T = 1} && t_1 && t_2
		\arrow[no head, from=2-3, to=1-3]
		\arrow[no head, from=2-3, to=1-5]
		\arrow[no head, from=2-5, to=1-3]
		\arrow[no head, from=2-5, to=1-5]
	\end{tikzcd},\] which yields a candidate for a bowtie. To conclude the proof, it suffices to show that there is no element $w\in [1,c]_T$ such that $\ell_T(w) = 2$ and $t_1, t_2 \leq_T w$. This will show that the quadruple $(t_1, t_2, w_1, w_2)$ is indeed a bowtie in $[1,c]_T$. 
	
	Assume for contradiction that such an element exists. There are thus $t, t'\in T$ such that $t_1, t_2 \leq_T t t'=w \leq_T c$. In particular, by Lemma~\ref{lem_omega_2}, we have $\omega_c( \alpha_t, \alpha_{t'}) = - 2 B(\alpha_t, \alpha_{t'})$.
	
	Since the standard parabolic subgroups $W_1$ and $W_2$ are of odd dihedral type and $t_1, t_2$ are their longest elements, there are $\lambda, \mu \geq 1 $ such that $\alpha_{t_1} = \lambda (\alpha_{s_1} + \alpha_{s_2})$ and $\alpha_{t_2} = \mu (\alpha_{s_1} + \alpha_{s_2})$. More precisely we have $\frac{1}{\lambda} = 2 \sin(\frac{\pi}{2 (2k + 1)})$ and $\frac{1}{\mu} =2 \sin(\frac{\pi}{2 (2k' + 1)})$. 
	
	We calculate \begin{align*}
	 B(\alpha_{t_1}, \alpha_{t_2}) & = \lambda \mu B(\alpha_{s_1} + \alpha_{s_2}, \alpha_{s_3} + \alpha_{s_4}) = \lambda \mu (B(\alpha_{s_1}, \alpha_{s_4}) +B(\alpha_{s_2}, \alpha_{s_3}) ) \\ &= - \lambda \mu 2 \cos( \pi / \ell) \leq - 2 \cos(\pi / \ell) \leq -1,
	\end{align*}
	We thus have that $\langle t_1, t_2 \rangle$ is an infinite dihedral reflection subgroup. 
	
	We also have 
	\begin{align*}
		\varphi_c(\alpha_{t_1}, \alpha_{t_2}) & = \lambda \mu \varphi_c(\alpha_{s_1} + \alpha_{s_2}, \alpha_{s_3} + \alpha_{s_4}) = \lambda \mu (\varphi_c(\alpha_{s_1}, \alpha_{s_4}) +\varphi_c(\alpha_{s_2}, \alpha_{s_3}) ) \\ &= \lambda \mu 2 B(\alpha_{s_2}, \alpha_{s_3})= -2 \lambda \mu \cos(\pi / \ell),
	\end{align*}
	and similarly we get $\varphi_c(\alpha_{t_2}, \alpha_{t_1}) = \lambda \mu 2 B(\alpha_{s_4}, \alpha_{s_1})= -2 \lambda \mu \cos(\pi / \ell) = \varphi_c(\alpha_{t_1}, \alpha_{t_2})$. This yields $\omega_c( \alpha_{t_1}, \alpha_{t_2}) = 0$. 
	
	Let $q, q'$ denote the canonical generators of the unique maximal dihedral reflection subgroup $W'$ of $W$ containing $t_1$ and $t_2$. This subgroup in particular contains $\langle t_1, t_2 \rangle$ which is infinite, it is thus infinite dihedral as well. Moreover, by~\cite[Remark (3.2)]{Dyer_Bruhat}, this subgroup $W'$ has as reflections precisely those $r\in T$ such that $\alpha_r \in \mathbb{R} \alpha_q \oplus \mathbb{R} \alpha_{q'}$. Since $$\mathbb{R} \alpha_t \oplus \mathbb{R} \alpha_{t'} = \M(w) = \mathbb{R} \alpha_{t_1} \oplus \mathbb{R} \alpha_{t_2} = \mathbb{R} \alpha_{q} \oplus \mathbb{R} \alpha_{q'},$$ we get that $t, t'\in W'$. 
	
	Now by~\cite[Proposition 4.1]{RS}, if $\omega_c(\alpha_q, \alpha_{q'}) \neq 0$, then $\omega_c(\alpha_{r}, \alpha_{r'}) \neq 0$ for every pair $r, r'$ of distinct reflections in $W'$. Since, as computed above, we have $\omega_c(\alpha_{t_1}, \alpha_{t_2}) = 0$, this forces $\omega_c( \alpha_q, \alpha_{q'}) = 0$, and hence again by~\cite[Proposition 4.1]{RS}, we get that $\omega_c( \alpha_t, \alpha_{t'}) = 0$. But since $\omega_c( \alpha_t, \alpha_{t'}) = - 2 B(\alpha_t, \alpha_{t'})$ we conclude that $tt'=t't$. This is a contradiction, as $t$ and $t'$ are distinct reflections in $W'$ which is infinite dihedral. 
	
	We thus have that $(t_1, t_2, w_1, w_2)$ is a bowtie in $[1,c]_T$, and hence, that $[1, c]_T$ is not a lattice.   
	\end{proof}

	Using Lemma~\ref{lem_parab_fail} we immediately deduce: 
	
	\begin{cor}\label{cor_fail}
	Let $(W,S)$ be a Coxeter system containing as a standard parabolic subgroup a Coxeter system $(W', S')$ satisfying the assumptions of Theorem~\ref{thm_fail_lattice}. Let $c$ be a Coxeter element in $W$ such that its restriction $c'$ to $W'$ satisfies the assumptions of Theorem~\ref{thm_fail_lattice}. Then $[1,c]_T$ fails to be a lattice. 
	\end{cor}
	
	\section{Conjectures and questions}
	
	In this section, we collect some existing questions or conjectures related to the material introduced in this paper, and also formulate some new ones. 
	
	\subsection{On the properties of elements of $[1,c]_T$}
	
	The following two questions are classical in the field (see for instance~\cite[Questions 3.1 to 3.3]{Paolini_survey}):
	
	\begin{question}\label{conj_cox_el}
	Let $(W,S)$ be a Coxeter system and $c\in W$ a Coxeter element. Let $x\in [1,c]_T$. Is $x$ a Coxeter element in $P(x)$? 
	\end{question}
	
	A positive answer to this question would imply a positive answer to the following one:
	
	\begin{question}\label{conj_hur_trans_prefix}
		Let $(W,S)$ be a Coxeter system and $c\in W$ a Coxeter element. Let $x\in [1,c]_T$. Is the Hurwitz action transitive on the $T$-reduced expressions of $x$?
	\end{question}
	
	\subsection{On the maps}
	
	\begin{conjecture}\label{conj_isom_3maps}
	Let $(W,S)$ be a Coxeter system and $c\in W$ a Coxeter element. Both maps $x \mapsto P(x)$ and $x\mapsto T(x)$ from $[1,c]_T$ to $\mathrm{RS}(W)$ and $\mathcal{P}(T)$ are isomorphisms of posets onto their images. 
	\end{conjecture}
	
	By Proposition~\ref{prop_3_maps}, the conjecture holds if $W$ has rank at most three.
	
	\subsection{On the subgroups $P(x)$}
	
	Little is known on the structure of the reflection subgroups $P(x)$ in the general case. Dyer has shown the following, which generalizes the existence of maximal dihedral reflection subgroups: 
	
	\begin{theorem}[{\cite[Theorem~1.8~(1)]{Dyer_low}}]
	Let $(W,S)$ be a Coxeter system. Let $t_1, t_2, \dots, t_k\in T$ such that $\alpha_{t_1}, \alpha_{t_2}, \dots, \alpha_{t_k}$ are linearly independent. Let $W'$ be the rank $k$ reflection subgroup $\langle t_1, t_2, \dots, t_k \rangle$. Then the set $\mathcal{R}(W')$ of reflection subgroups of $W$ which have rank $k$ and contain $W'$ has a maximum element $c(W')$. 
	\end{theorem}

	\begin{conjecture}\label{px_1}
	Let $(W,S)$ be a Coxeter system and $c\in W$ a Coxeter element. Let $x\in [1,c]_T$. Let $t_1 t_2 \cdots t_k$ be a $T$-reduced expression of $x$. Then $P(x) = c(\langle t_1, t_2, \dots, t_k\rangle)$.
	\end{conjecture}
	
	\begin{conjecture}\label{px_2}
	Let $(W,S)$ be a Coxeter system and $c\in W$ a Coxeter element. Let $x\in [1,c]_T$. For every $w\in P(x)$, the $T$-reduced expressions of $w$ in $W$ have all their factors in $P(x)$.  
	\end{conjecture}
	
	By Lemma~\ref{lemm_p_prop}, Conjecture~\ref{px_1} holds if $\ell_T(x)=2$. Conjecture~\ref{px_2} is also easily deduced in this case: elements of $P(x)$ are all products of at most two reflections of $P(x)$. If $w = t t' = q q'$ with $t, t'\in P(x)$, then $q, q'$ also lie in $P(x)$ since it is maximal dihedral (see for instance~\cite[Lemma 1.7]{Gobet_max_dih}). Both conjectures also hold for standard parabolic Coxeter elements $c' \leq_T c$, and when $P(x)$ is finite (in both cases $P(x)$ is parabolic). 
	
	The last conjecture is motivated by the following observations: when $W$ is finite, the subgroups $P(x)$ are parabolic, while if $W$ is infinite this is not the case in general (see Example~\ref{ex_nonparab}). It still seems natural to expect that they have properties similar to parabolic subgroups. Recall that parabolic subgroups are stable by intersection. For arbitrary $W$, the family of reflection subgroups with the property that for every element of such a group, all factors of the $T$-reduced expressions (in $W$) of the element are in the subgroup, is also stable by intersection. It might be a suitable generalization of parabolic subgroups. 
	
	It is not even clear that the intersection of two $P(x)$'s is a reflection subgroup of $W$. 
	
	\bigskip 
	
	\textbf{Acknowledgements and AI use.} The author thanks Jad Abou-Yassin and Jean-Yves Hée for stimulating discussions. He is partially supported by the ANR (Agence Nationale de la Recherche) project CaGeT (ANR-25-CE40-4162). AI was used to search the literature and check some computations from examples (free versions of ChatGPT and Claude). Everything was carefully checked after using these tools. The ideas, results and proofs come from the author, who takes full responsibility for their correctness.


\begin{thebibliography}{10}
		
		\bibitem{Abouyassin} J.~Abou-Yassin, \textsl{A generalization in affine type $A$ of Coxeter sortable elements and Reading's bijection with noncrossing partitions}, preprint (2026), \url{https://arxiv.org/abs/2605.02668}. 

		\bibitem{AB} P.~Abramenko and K.S.~Brown, \textsl{Buildings}, Grad. Texts in Math., 248
		Springer, New York, 2008, xxii+747 pp.
	
	
	\bibitem{Arm} D.~Armstrong, \textsl{Generalized noncrossing partitions and combinatorics of Coxeter groups}, Mem. Amer. Math. Soc. {\bf 202} (2009), no. 949, x+159 pp.

		
		\bibitem{BDSW} B.~Baumeister, M.~Dyer, C.~Stump, P.~Wegener, \textsl{A note on the transitive Hurwitz action on decompositions of parabolic Coxeter elements}, \emph{ Proc. Amer. Math. Soc. Ser. B} {\bf 1} (2014), 149-154.
		
		\bibitem{trickles} P.~Bellingeri, E.~Godelle, L.~Paris, \textsl{Trickle groups}, preprint (2024). \url{https://arxiv.org/abs/2412.04932}. 
		
	\bibitem{Dual} D.~Bessis, \textsl{The dual braid monoid}, Ann. Sci. \'{E}cole Norm. Sup. {\bf 36} (2003), 647-683.
	
			\bibitem{Bessis_free} D.~Bessis, \textsl{A dual braid monoid for the free group}, J. Algebra {\bf 302} (2006), 275-309.	
		
	\bibitem{Biane} P.~Biane, \textsl{Some properties of crossings and partitions}, Discrete Math. {\bf 175} (1997), no. 1-3, 41-53.
	
	\bibitem{BB} A. Bj\"orner and F. Brenti, \textsl{Combinatorics of Coxeter groups},
	Graduate Text in Math. Springer, 2005.
	
	\bibitem{Bourbaki} N.~Bourbaki, \textsl{Groupes et algèbres de Lie}, Éléments de mathématique, Chapitres IV-VI, Paris, Hermann, 1968. 	
		
	\bibitem{Brady} T.~Brady, \textsl{A partial order on the symmetric group and new $K(\pi,1)$'s for the braid groups},
	Adv. Math. {\bf 161} (2001), no. 1, 20–40.	
		
	\bibitem{BW_ortho} T.~Brady and C.~Watt, \textsl{A partial order on the orthogonal group}, 
	Comm. Algebra {\bf 30} (2002), no. 8, 3749--3754.
		
					\bibitem{BW_geom} T. Brady and C. Watt, \textsl{$K(\pi,1)$'s for Artin groups of finite type}, Geom. Dedicata {\bf 94} (2002), 225–250.
		
		\bibitem{BW_lattice} T. Brady and C. Watt, \textsl{Non-crossing partition lattices in finite real reflection groups}, Trans. Amer. Math. Soc. {\bf 360} (2008), no. 4, 1983–2005.
		
		\bibitem{BR} L.G.~Brestensky and N.~Reading, \textsl{Noncrossing partitions of an annulus},
		Comb. Theory {\bf 5} (2025), no. 1, Paper No. 12, 49 pp.	
		
		\bibitem{BH_a} B.~Brink and R.B.~Hohwlett, \textsl{A finiteness property and an automatic structure for Coxeter groups}, Math. Ann. {\bf 296} (1993), no. 1, 179–190.	
	
		
		\bibitem{BH} B.~Brink and R.~Howlett, \textsl{Normalizers of parabolic subgroups in Coxeter groups}, Invent. Math. {\bf 136} (1999), no. 2, 323–351.
		
			\bibitem{Carter} R.W.~Carter, \textsl{Conjugacy Classes in the Weyl
			group}, Compositio Math. {\bf 25} (1972), 1-59.
		
			\bibitem{Garside}
		P.~Dehornoy, F.~Digne, D.~Krammer, E.~Godelle, and J.~Michel.
		\textsl{Foundations of Garside theory}, Tracts in Mathematics {\bf 22}, Europ.\ Math.\ Soc.\
		(2015).
		
		\bibitem{Deodhar} V.V.~Deodhar, \textsl{A note on subgroups generated by reflections in Coxeter groups},
		Arch. Math. (Basel) {\bf 53} (1989), no. 6, 543–546. 
		
		
				\bibitem{DPS} E.~Delucchi, G.~Paolini, and M.~Salvetti, \textsl{Dual structures on Coxeter and Artin groups of rank three}, to appear in Geometry \& Topology (2023).
		
		
		\bibitem{Dig1} F.~Digne, \textsl{Présentations duales des groupes de tresses de type affine $\tilde{A}$}, \textit{Dual presentations of braid groups of affine type $\tilde{A}$}, Comment. Math. Helv. {\bf 81} (2006), no. 1, 23-47.
		
		\bibitem{Dig2} F.~Digne, \textsl{A Garside presentation for Artin groups of type $\tilde{C}_n$}, Ann. Inst. Fourier {\bf 62} (2012), no. 2, 641-666.
		
		\bibitem{Dus} K.~Duszenko, \textsl{Reflection length in non-affine Coxeter groups}, Bull. Lond. Math. Soc. {\bf 44} (2012), no. 3, 571-577.
		
		\bibitem{Dyer_subgroups} M.J.~Dyer, \textsl{Reflection subgroups of Coxeter systems}, J. of Algebra {\bf 135} (1990), Issue 1, 57--73.
		
		\bibitem{Dyer_Bruhat} M.J.~Dyer, \textsl{On the Bruhat graph of a Coxeter system}, Compositio Math. {\bf 78} (1991), no. 2, 185-191.
		
		\bibitem{Dyer_refn} M.J.~Dyer, \textsl{On minimal lengths of expressions of Coxeter group elements as products of reflections}, Proc. Amer. Math. Soc. {\bf 129} (2001), no. 9, 2591-2595.
		
		\bibitem{Dyer_low} M.J.~Dyer, \textsl{$n$-low elements and maximal rank $k$ reflection subgroups of Coxeter groups}, J. Algebra {\bf 607} (2022), part B, 139–180.
		
		\bibitem{Gobet_sortable} T.~Gobet, \textsl{Dual Garside structures and Coxeter sortable elements}, J. Comb. Algebra {\bf 4} (2020), no. 2, 167–213.
		
		\bibitem{Gobet_max_dih} T. Gobet, \textsl{On maximal dihedral reflection subgroups and generalized noncrossing partitions}, Proc. Amer. Math. Soc. {\bf 152} (2024), no. 10, 4095–4101.
		
		\bibitem{HK} A.~Hubery and H.~Krause, \textsl{A categorification of non-crossing partitions}, J. Eur. Math. Soc. (JEMS) {\bf 18} (2016), no. 10, 2273–2313.
		
		\bibitem{Humphreys} J.~Humphreys, \textsl{Reflection groups and Coxeter groups},
		Cambridge Stud. Adv. Math., 29,
		Cambridge University Press, Cambridge, 1990. xii+204 pp.
		
		\bibitem{IS} K.~Igusa and R.~Schiffler, \textsl{Exceptional sequences and clusters}, J. Algebra {\bf 323} (2010), no. 8, 2183–2202.
		
		\bibitem{IT}  C.~Ingalls and H.~Thomas, \textsl{Noncrossing partitions and representations of quivers}, Compos. Math. {\bf 145} (2009), no. 6, 1533–1562.
		
		\bibitem{BLMC} J.B.~Lewis, J.~McCammond, T.K.~Petersen, P.~Schwer, \textsl{Computing reflection length in an affine Coxeter group}, Trans. Amer. Math. Soc. {\bf 371} (2019), no. 6, 4097–4127.
		
		
		\bibitem{McCammond} J.~McCammond, \textsl{Dual euclidean Artin groups and the failure of the lattice property},
		J. Algebra {\bf 437} (2015), 308-343.
		
		\bibitem{MS} J.~McCammond and R.~Sulway, \textsl{Artin groups of Euclidean type},
		Invent. Math. {\bf 210} (2017), no.1, 231–282.
		
		\bibitem{Obrien} S.~O'Brien, \textsl{The dual Artin isomorphism for Artin groups of XXL type}, preprint (2026), \url{https://arxiv.org/abs/2606.13296}. 
		
		\bibitem{Paolini_survey} G.~Paolini, \textsl{The dual approach to the $K(\pi, 1)$ conjecture},
			Geometric methods in group theory—papers dedicated to Ruth Charney, 177–201. Sémin. Congr., {\bf 34}, Société Mathématique de France, Paris, 2025.
		
		\bibitem{PS} G.~Paolini and M.~Salvetti, \textsl{Proof of the $K(\pi,1)$ conjecture for affine Artin groups}, Invent. Math. {\bf 224} (2021), no. 2, 487–572.
		
		\bibitem{Paris_irred} L.~Paris, \textsl{Irreducible Coxeter groups}, Internat. J. Algebra Comput. {\bf 17} (2007), no. 3, 427–447.
		
		\bibitem{Reading_shards} N.~Reading, \textsl{Noncrossing partitions and the shard intersection order}, J. Algebraic Combin. {\bf 33} (2011), no. 4, 483–530.
		
		\bibitem{Reading_sortable} N.~Reading, \textsl{Clusters, Coxeter-sortable elements and noncrossing partitions}, Trans. Amer. Math. Soc.  {\bf 359}  (2007),  no. 12, 5931-5958.
		
		\bibitem{RS} R.~Reading and D.~Speyer, \textsl{Sortable elements in infinite Coxeter groups}, Trans. Amer. Math. Soc. {\bf 363} (2011), no. 2, 699–761.
		
		\bibitem{Speyer} D.~Speyer, \textsl{Powers of Coxeter elements in infinite groups are reduced}, Proc. Amer. Math. Soc. {\bf 137} (2009), no. 4, 1295–1302.

		
	\end{thebibliography}
	\end{document}